\documentclass[12pt,a4paper]{amsart}

\usepackage[T1]{fontenc}
\usepackage[utf8]{inputenc}
\usepackage[english]{babel}
\usepackage{amsfonts, amssymb, amsmath, amsthm}
\usepackage[all]{xy}
\usepackage{enumerate}

\usepackage{graphicx}
\usepackage{psfrag}
\usepackage{geometry}
\usepackage[colorlinks=true,pdfstartview=FitV,linkcolor=blue,citecolor=green,urlcolor=black,filecolor=magenta]{hyperref}
\usepackage{verbatim}

\newtheorem{theorem}{Theorem}[section]
\newtheorem{definition}[theorem]{Definition}
\newtheorem{prop}[theorem]{Proposition}
\newtheorem{lemma}[theorem]{Lemma}
\newtheorem{cor}[theorem]{Corollary}

\newcommand{\Aa}{{\mathcal A}}
\newcommand{\Bb}{{\mathcal B}}
\newcommand{\Cc}{{\mathcal C}}

\newcommand{\Gg}{{\mathcal G}}
\newcommand{\Hh}{{\mathcal H}}

\newcommand{\Kk}{{\mathcal K}}

\newcommand{\Mm}{{\mathcal M}}

\newcommand{\Zz}{{\mathcal Z}}
\newcommand{\Xx}{{\mathcal X}}

\newcommand{\id}{\mathrm{id}}

\newcommand{\CM}{{\mathbb C}}

\newcommand{\FM}{{\mathbb F}}
\newcommand{\HM}{{\mathbb H}}

\newcommand{\NM}{{\mathbb N}}

\newcommand{\RM}{{\mathbb R}}

\newcommand{\SM}{{\mathbb S}}
\newcommand{\TM}{{\mathbb T}}

\newcommand{\Z}{{\mathbb Z}}

\newcommand{\R}{{\mathbb R}}

\newcommand{\C}{{\mathbb C}}

\newcommand{\rs}{{\mathscr R}}

\newcommand{\ind}{\mbox{\rm ind}}

\newcommand{\cC}{\mathfrak C}
\newcommand{\cc}{\mathfrak c}
\newcommand{\CA}{$C^*$-algebra}
\newcommand{\simh}{{\sim_h}}
\newcommand{\simhe}{{\sim_h^e}}
\newcommand{\Us}{{\mathfrak S}}

\newcommand{\tp}{\mathfrak t}
\newcommand{\fp}{\phi}

\newcommand{\rc}{\mathfrak{l}}
\renewcommand{\rs}{\mathfrak r}
\renewcommand{\ss}{{\mathfrak s}}
\newcommand{\rh}{{\mathfrak h}}

\newcommand{\rt}{{\mathfrak t}}
\newcommand{\rp}{\mathfrak p}
\newcommand{\rf}{\mathfrak f}

\newcommand{\DK}{D\!K}
\newcommand{\Ad}{\mathrm{Ad}}

\newcommand{\st}{\mathrm{st}}

\newcommand{\stot}{\id\!\otimes\!\st}
\newcommand{\stev}{\st}

\newcommand{\RA}{$C^{*,r}$-algebra}

\newcommand{\Gr}{{G}}
\newcommand{\spec}{\mathrm{spec}}

\title{On the $C^*$-algebraic approach to topological phases for insulators}

\author{Johannes Kellendonk}

\address{Univerisit\'{e} de Lyon, Universit\'{e} Claude Bernard Lyon 1, Institute Camille Jordan, CNRS UMR 5208, 69622 Villeurbanne, France}
\email{kellendonk@math.univ-lyon1.fr}

\date{\today}

\begin{document}

\begin{abstract}
The notion of a topological phase of an insulator is based on the concept of homotopy between Hamiltonians. It therefore depends on the choice of a topological space to which the Hamiltonians belong. We advocate that this space should be the $C^*$-algebra of observables.  We relate the symmetries of insulators to graded real structures on the observable algebra and classify the topological phases 
using van Daele's formulation of $K$-theory. 
This is related but not identical to Thiang's recent approach to classify topological phases by $K$-groups in Karoubi's formulation.
\end{abstract}

\maketitle

\section{Introduction}

An insulator is described by a Hamiltonian with a gap in its spectrum.
It may be subject to certain symmetries of order two.
Inside a symmetry class insulators are grouped into topological phases. 
In fact, two insulators of the same symmetry type are in the same topological phase if their Hamiltonians can be continuously deformed into each other while preserving the symmetry and the gap.
The delicacy of this notion lies in the question in which background space and with respect to which topology should continuous paths of Hamiltonians be considered. 

Often in the literature this question is considered from a slightly different point of view, namely  from the point of vector bundles 
(this goes back to the early work on the Quantum Hall Effect \cite{TKNN}, for recent mathematically advanced developments see  
\cite{Gomi,Kennedy}): It is supposed that the Hamiltonian fibers over the space of (quasi) momenta, called the Brillouin zone. The occupied states being separated by a spectral gap from the conducting states they define a vector bundle over the Brillouin zone. Topological phases are then defined as the homotopy classes of these vector bundles with certain symmetries (or, depending on the point of view, pseudo-symmetries). Here homotopy between two vector bundles is considered inside a common larger bundle. This would be satisfactory if the assumption that the Hamiltonian fibers over momenta would always be satisfied. But disorder, which is indispensible to understand certain aspects of topological quantisation, distroys that fibration.  

Freed and Moore propose an answer which is inspired by the fact that insulators are usually described by Hamiltonians which are represented on some Hilbert space $\Hh$. 
They consider continuous paths in $\Bb(\Hh)$ or, more generally taking into account a variation of the Hilbert spaces, continuous paths in some bundle whose fibers are operators on some Hilbert space. The subtle issue is the topology to be considered. Equipped with the norm topology $\Bb(\Hh)$ is non-separable (if the dimension is infinite) and its $K$-theory trivial. This makes $K$-theory rather useless. Indeed, Freed and Moore propose to consider the compact open topology (topology of uniform convergence on compact subsets) which is close to the strong topology. This however makes index theory trivial (see their remark in \cite{Moore}). Since index theory has proven very successful
applied to insulators \cite{Schuba,GSchuba} we would not like to dispense with it and therefore find the compact open topology unsuited. Also the weak topology destroys all topological information, as there would be too many continuous paths. Thus we need something better.

In the spirit of describing physical systems (including those with disorder) by their \CA\ of observables and their topology by the non-commutative topology of \CA s it is much more natural to consider an approach where homotopy is considered in a \CA. The difficulty here is to translate the notion of symmetry type from a symmetry which is represented on some Hilbert space to a symmetry on a \CA. Once this is achieved we can define homotopy in the $C^*$-observable algebra
using the $C^*$-topology. This topology has proven useful in $K$-theory and in index theory. The point is that a norm closed subalgebra of $\Bb(\Hh)$ may have much more structure than all of $\Bb(\Hh)$. In particular, in the examples we have in mind the \CA s are separable.

In some sense this is the approach proposed by Thiang \cite{Thiang}. He takes the symmetry into account by considering graded, twisted crossed products of the \CA\ by actions which 
may partly be given by anti-linear automorphisms. While this is very nice, I have difficulties to understand the physical interpretation. Neither is it properly argued for why this way of incorporating the symmetries is the correct one, nor precisely explained how the elements of the $K$-group he uses are related to insulators. We come back to Thiang's approach after providing a rough overview of our approach which is based on van Daele's formulation of $K$-theory. Indeed, van Daele's formulation of $K$-theory is so close to the problem of classifying topological phases that it could have been invented for this purpose.


\subsection{Short description of our approach}
We start somewhat abstractly by cnsidering the (complex) \CA\ of observables $A$. Specific examples exist, e.g.\ for the quantum Hall effect, or a quasicrystal, or other kinds of disordered systems described in the one-particle approximation.
The observable algebra depends on the long range order properties of potential of the material; like periodicity or quasiperiodicity or randomness.
It also depends on the external magnetic fields. 
But for the general theory we don't need to know any details about $A$. Only one property is us important:  there is a preferred family of representations of $A$ so that the operation of complex conjugation on the Hilbert space of such a representation may give rise to a natural real structure on $A$. This real structure will serve as a reference real structure. 

An abstract insulator is a self adjoint element $h\in A$ which has a gap in its spectrum at the Fermi energy which we fix to be at $0$. In other words, the set of insulators with observable algebra $A$ is the set of self-adjoint invertible elements of $A$ which we denote $A^{s.a.}_{inv}$. Topological phases are defined via homotopy in $A^{s.a.}_{inv}$. If no symmetries are present, this is it. The topological phases are classified by $\pi_0(A^{s.a.}_{inv})$, the homotopy classes in $A^{s.a.}_{inv}$. An abstract functorial machinery applied to 
$\pi_0(A^{s.a.}_{inv})$ turns it into a group, the $K_0$-group of the complex \CA\ $A$.

What about symmetries? In the \CA ic approach symmetries are implemented by automorphisms of the algebra. A chiral symmetry is therefore implemented by a complex linear automorphism of order two $\gamma$ on $A$ and an insulator with chiral symmetry is an element $h\in A^{s.a.}_{inv}$ which satisfies $\gamma(h) = -h$. Automorphisms of order two are the nothing else than gradings: invariant elements are even and anti-invariant elements odd. 
The topological phases of chiral insulators are therefore classified by the homotopy classes of 
{\em odd} elements  in $A^{s.a.}_{inv}$. The abovementionned machinery applied to these classes yields a group which is the $K_1$-group of the complex graded algebra $(A,\gamma)$. 

Time reversal is implemented by an anti-linear automorphism of order two $\rt$ on $A$ and an insulator with time reversal symmetry is a self adjoint invertible element of $A^\rt$, the $\rt$-invariant elements of $A$. These elements form a real \CA. Now to describe topological phases of time reversal invariant insulators we have to look at homotopy in the set of self adjoint invertible elements the real algebra $A^\rt$, or, if there is also a chiral symmetry $\gamma$, in its subset of odd elements. The machinery of making a group now yields the $K_0$ group of the real algebra $A^\rt$ or, in the presence of a chiral symmetry, the $K_1$-group of the graded real algebra $(A^\rt,\gamma)$. Equivalently, instead of speaking about real algebras one may formulate these matters also in \RA s (Real \CA s)
$(A,\rt)$ or $(A,\gamma,\rt)$.

Also particle hole exchange is implemented by an anti-linear automorphism of order two $\rp$ on $A$. But now we are interested in insulators which transform as $\rp(h) = -h$ under this exchange. Thus $h$ is not in $A^\rp$. There is a trick: replace $A$ by $A\otimes\C l_1$, the complex Clifford algebra of one generator and extend $\rp$ to $\rp\otimes\rc_{0,1}$ where $\rc_{0,1}$ is complex conjugation on $\C l_1\cong \C\oplus\C$ followed by  exchange of the summands. The inclusion $h\mapsto (h,-h)$ then maps an insulator which transforms as $\rp(h) = -h$  to a $\rp\otimes\rc_{0,1}$-invariant element and thus turns particle hole exchange into a genuine symmetry. Developing this further one finds the $K_{2}$-group of $A^\rp$ as group associated with topological phases of insulators with particle hole symmetry.

We study these $K$-groups using van Daele's formulation of  $K$-theory for graded real or complex Banach algebras \cite{Daele}. We like his formulation
because of its simplicity and similarity with insulators: In an only slightly oversimplified way one may say that the elements of van Daele's $K$-group {\em are} insulators with symmetries.
Furthermore, it treats even and odd $K$-groups in a unified manner. 
The recent formulation of $KO$-theory by unitaries \cite{Loring,GSchuba} can be easily obtained from van Daele's formulation.

We have summarized so far the description of five cases, the results which the reader can find in Table~\ref{tab-1}. The famous classification in the $8+2$-fold way will be obtained under further assumptions. 

The first of these assumptions is that the grading is an {\em inner} grading.
This means that $\gamma$ (the chiral symmetry) is given by conjugation with a so-called grading operator $\Gamma$, a self-adjoint unitary of $A$ (or the multiplier algebra of $A$ in case the latter is not unital).  
Given that, we have to ask how $\Gamma$ transforms under time reversal. 
If the only central projections of $A$ are $0$ and $1$ (the Gelfand spectrum of the center of $A$ is connected) then
there are only the two possibilities that $\rt(\Gamma) = \pm \Gamma$. These two possibilities (called {\em real inner} and {\em imaginary inner}) split the case of insulators with time reversal and chiral symmetry into two subcases. 

The second assumption is that the real structures (time reversal or particle hole exchange) are related in a specific way to the reference structure $\rf$ on the observable algebra coming from complex conjugation in its preferred representations. What we mean is that $\rt$ (or $\rp$) is $\rf$ followed by a conjugation with a unitary $\Theta$. 
The sign of $\rf(\Theta)\Theta$ can be interpreted as the parity (even / odd) of time reversal (or particle hole exchange).

Taking the above two times two possibilites
into consideration leads to the finer classification presented in Table~\ref{tab-or} and Table~\ref{tab-4}. Applied to the strong invariants of a tight binding model (Table~\ref{tab-si})
or to
%
the observable algebra $A$ which we get if the disorder space $\Omega$ is reduced to one point and there is no external magnetic field, these tables 
correspond to the famous periodic table established in \cite{Schnyder,Kitaev}. 

\subsubsection{Thiang's approach}
We summarize very shortly the elements of Thiang's work \cite{Thiang} which have direct relevance to ours. The symmetries of the insulator determine a subgroup of what 
Freed and Moore call the $CT$-group \cite{Moore}. This subgroup 
$\Gg$ comes with 
two group homomorphisms $\phi,c:\Gg\to \Z_2$. Thiang implements this symmetry on the \CA ic level by considering the twisted crossed product $A=\Aa\rtimes_{\alpha,\sigma}\Gg$ of a \CA\ $\Aa$ with $\Gg$. $\Aa$ ought to depend on the additional  structures, like periodicities or a disorder space and therefore seems to play a role comparable to the \CA\ of observables, although it has to be a real \CA\ in case there are real symmetries.
The homomorphism $\phi$ encodes whether a group element acts complex- or anti-linearly on $\Aa$. This together with the parity of the real symmetries is taken into account by the action $\alpha$ and the twisting cocycle $\sigma$. The other homomorphism $c$ is interpreted as a grading on $\Gg$ and therefore induces a grading on the crossed product $\Aa$. 
Thiang's and our approach share therefore a variety of features, but there are subtle differences. Thiang's grading on the algebra comes from a grading of the symmetry group. In particular, the unitary corresponding to particle hole exchange is an odd element of his algebra and the spectrally flattened insulator is represented as a grading operator, hence an even element.
In our approach, the grading is given by the chiral symmetry, if such a symmetry is present, and the insulator appears as an odd element.

Thiang  defines his $K$-group by considering homotopy classes of gradings on $\Aa$-modules. More precisely, he considers finitely generated projective graded $\Aa$-modules $(W,\Gamma)$. Here $W$ is a finitely generated projective $\Aa$-module
and  $\Gamma$ a grading operator on $W$ such that the $\Aa$-action becomes a graded action (the suggested interpretation is that $\Gamma$ is
the spectrally flattened Hamiltonian). The elements generating Thiang's $K$-group are homotopy classes of triples $(W,\Gamma,\Gamma_0)$, 
where $\Gamma$ and $\Gamma_0$ are two grading operators 
which are compatible with the graded action on $\Aa$ and homotopy is considered w.r.t.\ the norm topology on the endomorphism ring of bounded linear maps on $W$ ($W$ is a Banach space whose topology ultimately comes from the norm topology on $A$).  Thiang's  answer to the question posed above after the background space in which homotopy is considered 
is therefore the following: The space is the set of grading operators on $W$ equipped with the norm topology. I find this construction far removed from the physical problem. 
Unfortunately \cite{Thiang} offers little explanation as to
why $W$ with its Banach topology is
the correct object to look at from the point of view of physics.

A better understanding of the precise relationship between Thiang's and our approach is work in progress.

\section{\CA ic approach to topological insulators}
The \CA ic approach to topological insulators is based on the non commutative topology of  the \CA\ of observables $A$ of the insulator in the one-particle approximation. The construction of such a \CA\ for aperiodic and in particular disordered media was developped by Jean Bellissard in the eighties, inspired by the use of covariant families of (one-particle) Schr\"odinger operators in disordered systems. 
Bellissard used non-commutative topology to study the Gap-Labelling \cite{Bel86, BHZ}
and the Integer Quantum Hall Effect \cite{BES} (which is a topological insulator without symmetry). Most of what follows won't need the details of the construction, but only the discussion of two issues: the construction of a reference real form $\rf$ on $A$ and the triviality of the center of $A$. 

\subsection{Observable algebra}
A presentation of the details of the construction of the observable algebra can be found in \cite{BHZ}. We refer to \cite{MPR} for details on the twisting related to variable magnetic fields but also to \cite{procBucharest} for the link with non-commutative topology. 

A particular feature of the construction is that the medium is not described by a single configuration of its atoms, but that there is a whole set $\Omega$ of such configurations, the so-called hull of the medium, which is closed under the operation of shifting the configuration in Euclidean space $\R^d$.
This can be understood in a probabilistic way, because we can say only with a certain probability that a specific configuration is the actual one, the (shift invariant ergodic) probability measure on $\Omega$ being part of the physical data (it can be seen as a choice of phase) \cite{BHZ}. 
Properties of the medium are then obtained as expectation values w.r.t.\ to the probability measure and so one does not study a single Hamiltonian but a whole family of Hamiltonians 
$\{H_\omega\}_{\omega\in\Omega}$, for each configuration one. The different Hamiltonians are related through a so-called covariance relation.

In our context it is very important is that $\Omega$ can be equipped with a topology in which it is compact, and with respect to which the shift action $\alpha$ of the group of translations 
$\R^d$ is continuous. There are several ways to understand this topology one  being by interpreting an element $\omega\in\Omega$ as a measure on $\R^d$ and using the vague topology on measures \cite{BHZ}. For instance, one can attach to the position of an atom a Dirac measure and then take for $\omega$ the sum over all these Dirac measures, possibly weighted to take the atomic type into account.

Internal degrees of freedom like spin or pseudo-spin can be incorporated by tensoring the above algebra with a finite dimensional algebra $M_n(\C)$. 

A configuration dependent external magnetic field $B_\omega$ can be taken care of
by means of a $2$-cocycle
$\sigma:\R^d\times \R^d \to C(\Omega,S^1)$  
$$\sigma(x,y)(\omega) = \exp(-i\Phi_\omega(0,x,x+y))$$ 
where $\Phi_\omega(a,b,c)$ is the flux of the magnetic field $B_\omega$ through the (oriented) triangle with corners $\{a,b,c\}$. For this we require that the dependence of the magnetic field on the configuration $\omega$ is such that the fluxes $\Phi_\omega(0,x,x+y)$
are continuous in $\Omega$. 

Taking into account the above one is led to define the  \CA\ of observables as the twisted crossed product algebra
$$A =  C(\Omega,M_n(\C))\rtimes_{\alpha,\sigma} \R^d.$$ 
Here the action on $f\in\C(\Omega,M_n(\C))$ is by pull back $\alpha_x(f)(\omega) = f(\alpha^{-1}_x(\omega))$. $A$ is the completion of the convolution algebra $L^1(\R^d,C(\Omega,M_n(\C)))$ which has convolution product 
$$F_1 * F_2 \,(x) = \int_\R^d F_1(y))\, \alpha_y(F_2(x-y))\,\sigma(x,x-y)dy$$
and $*$-structure $F^*(x) = \alpha_x(F(-x))^*$. 

The algebra $A$ is the abstraction of families of covariant operators like a group is the abstraction of its defining representations. Indeed, given any twisted crossed product $B\rtimes_{\beta,\tau} G$, any representation $\rho$ of $B$ on some Hilbert space $\Hh$ 
induces in a canonical way a representation $\ind[\rho]$ of $B\rtimes_{\beta,\tau} G$ on $L^2(G,\Hh)$ \cite{PackerRaeburn}. Following this induction principle we can assign to each configuration $\omega\in \Omega$ a representation $\pi_\omega$ of $A$ on $L^2(\R^d,\C^n)$ by combining the representation $\rho_\omega$ of $C(\Omega,M_n(\C))$ with the representation $T_\omega$ of $\R^d$ on $L^2(\R^d,\C^n)$ \cite{PackerRaeburn}
\begin{eqnarray*}
\rho_\omega(f)\Psi(x) &=& f(\alpha^{-1}_x(\omega))\Psi(x) \\
T_\omega(a)\Psi(x) & = & \sigma(x,a)(\omega)\Psi(x+a).
\end{eqnarray*}
The ($M_n(\C)$-valued) integral kernel of $\pi_\omega(F)$,
$F\in L^1(\R^d,C(\Omega,M_n(\C)))$, 
is then given by 
\begin{equation}\label{eq-repr}
\left(\pi_\omega(F)\right)_{xy} =  \sigma(x,y-x)(\omega)\, F(y-x)(\alpha^{-1}_{x}(\omega)). 
\end{equation}
The family $\{\pi_\omega(F)\}_{\omega\in\Omega}$ satisfies the covariance relation 
$$\pi_{\alpha^{-1}_x(\omega)}(F) = U_\omega(x) \pi_\omega(F)U_\omega(x)^{-1}$$ where 
$$U_\omega(a) \Psi(x) = \sigma(a,x)(\omega)\, \Psi(x+a).$$
Thus a covariant family of bounded Hamiltonians (or bounded functions of Hamiltonians, like resolvants or heat kernels) corresponds to the family $\{\pi_\omega(h)\}_{\omega\in\Omega}$ for some self adjoint element $h\in A$. With this in mind a Hamiltonian describing particle motion in the medium will be considered as an element of $A$.

While $A$ has many more representations (we even have to expect that its representations are unclassifiable) the above representations are from the physical perspective the natural ones. We will use these representations lateron to define a reference real structure on $A$.

\subsubsection{Tight binding approximation}
The above description of the observable algebra is appropriate to describe covariant families of Hamiltonians which are differential operators on $\R^d$. Often
Hamiltonians in the tight binding approximation are used, as they are technically simpler and the corresponding observable algebra unital. 

The essential difference in the construction is the replacement of derivatives by difference operators. This is obtained by restricting the $\R^d$ shift action to the so-called canonical transversal $\Xi$ of $\Omega$. If we view the elements of $\Omega$ as point sets then, fixing an origin $0\in\R^d$, $\Xi$ is the subset of point sets $\omega\in\Omega$ which contain the origin.
We make the common assumption that $\Omega$ has finite local complexity, that is, given $R>0$ the set $\{B_R(0)\cap \omega:\omega\in\Xi\}$ of possible intersections of the $R$-ball at $0$ with the configurations $\omega\in\Omega$, is finite (there are only finitely many local configurations of size $R$). Then $\Xi$ is a closed subset of $\Omega$ whose topology admits a base of clopen sets which are in one to one correspondance to the local configurations of size $R$. Given a discrete subset $P$ of $B_R(0)$ the set $U_P=\{\omega\in\Xi:B_R(0)\cap \omega=P\}$ is clopen and these sets define the topology.
Restricting the $\R^d$ action on $\Omega$ to $\Xi$ leads to an \'etale groupoid
$\Gg\subset \Omega\times\R^d$ and the observable algebra of the tight binding approximation is the groupoid \CA\ \cite{Ke1}, possibly twisted by a cocycle to incorporate the magnetic field. 
We can dispense, however, to work with groupoids by reducing $\Xi$ further to a closed subset $X$ on which the $\R^d$-action induces a $\Z^d$ action. This is guaranteed by \cite{SW} and means effectively that we can identify in each configuration $\omega\in \Omega$ a (possibly distorted) lattice $\Z^d$ by inspection of the local configurations of a given finite size.
It follows that the observable algebra for the tight binding approximation is a crossed product
$$A = C(X,M_N(\C))\rtimes_{\alpha}\Z^d$$
where $X$ is a totally disconnected compact space on which $\Z^d$ acts continuously. 
Here we have excluded external magnetic fields\footnote{While there is no difficulty to include homogenuous external magnetic fields by means of cocycles, 
variable fields need more care in the tight binding approximation.}
since the examples we will discuss are invariant under a real symmetry and therefore incompatible with such fields, as we will argue below. 

A Hamiltonian in the tight binding approximation will thus be seen as a self adjoint operator $h\in C(X,M_N(\C))\rtimes_{\alpha}\Z^d$ and, as in the case of $\R^d$-actions, any configuration $\omega\in X$ gives rise to a representation $\pi_\omega$ of $C(X,M_N(\C))\rtimes_{\alpha}\Z^d$ on $\ell^2(\Z^d,\C^N)$ so that the $H_\omega=\pi_\omega(h)$ form a covariant family of operators describing the particle motion in the solid.

The simplest case is when $\Omega$ consists of the translates of a completely periodic configuration $\omega$, that is,  $\omega$ is the periodic repetition of a finite set of $K$ points
(the elementary cell). This yields the description of a crystal. Most models considered in the context of topological insulators are of that type.
In this case $\Omega$ is a $d$-torus and $X$ can be taken to be a single point. The translation action induces on $X$ the trivial action and $A = M_N(\C)\rtimes_{\id}\Z^d\cong C(\hat{\Z^d},M_N(\C))$. The isomorphism is given by the Fourier transform and 
$\widehat{\Z^d} = \R^d\slash \Gamma^{rec}$ is
the dual group to $\Z^d$  ($\Gamma^{rec}$ is the reciprocal of the periodicity lattice). Though this is also a $d$-torus it should not be confused with $\Omega$, as this $d$-torus is in momentum space and corresponds to the Brillouin zone.
A Hamiltonian $h\in A$ is thus a matrix valued function on $\widehat{\Z^d}$
$$ h(k) = \sum_{n\in\Z^d} t_n e^{2\pi  i n k} $$
$k\in \R^d\slash \Gamma^{rec}$. Since $X$ has only one point, the covariant family consists of one member only, namely $\pi_\omega(h) = \sum_{n\in\Z^d} t_n u^{n}$ where $u^n = u_1^{n_1}\cdots u_d^{n_d}$ are translation operators on $\ell^2(\Z^d,\C^N)$, notably  $u_i\Psi(n) = \Psi(n-e_i)$ is the translation in direction $i$.

\subsection{Disorder and triviality of the center}
In the periodic case (crystal) the observable algebra $A$, or its multiplier algebra, has a center which usually corresponds to the continuous functions over the Brillouin zone. 
One of the features of disorder is that it distroys the Brillouin zone. In the algebraic language
this means that center becomes trivial. We provide below a precise criterion for that.

The multiplier algebra of the \CA\ $A$, denoted $\Mm(A)$, is the smallest unital \CA\ which contains $A$ as an essential ideal. It is determined up to isomorphism and coincides with $A$ if $A$ has already a unit. 
We denote by $\Zz(A)$ the center of $\Mm(A)$. This is the set of elements in $\Mm(A)$ which commute with all elements of $A$ (and hence of $\Mm(A)$). $\Zz(A)$ is a unital commutative \CA\ and as such is isomorphic to $C(\Xx(A))$ for some compact Hausdorff space $\Xx(A)$, its so-called Gelfand spectrum.
Note that $\Xx(A)$ is connected if $\Zz(A)$ does not contain any non-trivial projection. Furthermore,  $\Zz(A)$ is trivial if and only if 
$A$ admits a faithful  irreducible representation. 
%
We will need the following result.\footnote{
A proof of the first assertion of the theorem in the case that $A$ is commutative, but the action is twisted, can be found in \cite{MPR}.  I am thankful to Siegfried Echterhoff for discussions about this result and bringing Thm.~1.7 of \cite{Echterhoff} to my attention.}

\begin{theorem}\label{thm-Echt}
Let $(A,G,\alpha)$ be a $C^*$-dynamical system where $A$ is a separable $C^*$-algebra and 
$\alpha$ a continuous action of a locally compact abelian group $G$. Let $\rho$ be a representation of $A$.
\begin{enumerate}
\item If the intersection ideal $\bigcap_{g\in G} \ker\rho\circ\alpha_g$ is the $0$ ideal then the induced representation $\ind[\rho]$ of $A\rtimes_\alpha G$ is faithful.
\item If the stabiliser $G_P = \{g\in G:\ker\rho\circ\alpha_g = \ker\rho\}$ is the trivial group then the induced representation $\ind[\rho]$ of $A\rtimes_\alpha G$ is irreducible.
\end{enumerate} 
\end{theorem}
\begin{proof}
(1) 
Since the induced representations $\ind[\rho]$ and $\ind[\rho\circ\alpha_g]$ of $\rho$ and of 
$\rho\circ\alpha_g$ are unitarily equivalent they have the same kernel and this kernel is also the kernel of the direct sum representation $\bigoplus_{g\in G} \ind[{\rho\circ\alpha_g}]$. Moreover 
$\bigoplus_{g\in G} \ind[{\rho\circ\alpha_g}]$ is the induced representation of 
$\bigoplus_{g\in G} \rho\circ\alpha_g$. As $\ker \bigoplus_{g\in G} \rho\circ\alpha_g = \bigcap_{g\in G} \ker\rho\circ\alpha_g$ the representation $\bigoplus_{g\in G} \rho\circ\alpha_g$
is faithful by assumption. Since induced representations of faithful representations are faithful we conclude that $\ker\ind[\rho]$ is trivial, that is, $\ind[\rho]$ is faithful.

(2) This is a direct application of Thm.~1.7 of \cite{Echterhoff} applied to the 
representation $\rho$.
\end{proof}
\begin{prop}
Let $(X,G,\alpha,\sigma)$ be a twisted dynamical system where $X$ is a compact metrisable space, $\alpha$ a continuous action of a locally compact abelian group $G$, and $\sigma$ a twisting cocycle. 
\begin{enumerate}
\item If $X$ contains a dense orbit then the center $\Zz(A)$ of the multiplier algebra of 
$A=C(X)\rtimes_{\alpha,\sigma} G$ has connected Gelfand spectrum. 
\item  If $X$ contains a free dense orbit then 
$\Zz(A)=\C$.  
\end{enumerate}
\end{prop}
\begin{proof}
(1) Let $p\in\Zz(A)$ be a non-zero projection and $p^\perp=1-p\neq 0$. Then $C(X)$ can be written as a direct sum of abelian sub-algebras $pC(X)\oplus p^\perp C(X)$. It follows that $X$ is the disjoint union of two compact subsets $X_p$ and $X_{p^\perp}$. Since $p$ is central, it is $G$-invariant and hence   $X_p$ and $X_{p^\perp}$ are $G$-invariant. But any dense orbit intersects both. Hence $p^\perp=0$.

(2) Let $x$ be a point in a free dense orbit of $X$. If $\sigma$ is trivial then we apply 
the last theorem to the evaluation representation $\rho=ev_x$. Clearly, the orbit being dense implies that the intersection ideal $\bigcap_{g\in G} \ker\rho\circ\alpha_g$ is trivial, and the orbit being free that the stabiliser $G_P = \{g\in G:\ker\rho\circ\alpha_g = \ker\rho\}$ is trivial. 
By Thm.~\ref{thm-Echt} the induced representation of $ev_x$ is a faithful irreducible representation of 
$C(X)\rtimes_\alpha G$.

Suppose now that the twisting cocycle $\sigma$ is non-trivial. By the Packer-Raeburn stabilisation trick \cite{PackerRaeburn} there exists an action $\tilde\alpha$ of $G$ on $\Kk\otimes C(X)$ such that, 
$\Kk\otimes C(X)\rtimes_{1\otimes\alpha,1\otimes \sigma} G$ is isomorphic to $\Kk\otimes C(X)\rtimes_{\tilde\alpha} G$ and $\tilde\alpha_g(\Kk\otimes I) = \Kk\otimes\alpha_g(I)$ for any ideal $P$ of $C(X)$. Applied to the prime ideal $P = \ker ev_x$ of $C(X)$ we therefore find that  the representation $\rho = \id\otimes ev_x$ satisfies the assumptions of Thm.~\ref{thm-Echt} and thus conclude that $\Zz(\Kk\otimes A)$ is trivial. Finally the multiplier algebras of $\Kk\otimes A$ and of $A$ have isomorphic center. 
\end{proof}

\begin{cor}
The \CA\ of the Bernoulli shift has trivial center.
\end{cor}
\begin{proof} The Bernoulli shift contains a free dense orbit.
\end{proof}
Since the Bernoulli shift is the prototype of a tight binding system with disorder we may expect that the observable algebra for a disordered system has trivial center. 


\subsection{Abstract insulators}\label{sec-2.3}
The observable algebra is a complex \CA. If it is not unital we add a unit to it to obtain the minimal unitisation. 
\begin{definition} Given an observable \CA\ $A$ an abstract insulator is a self-adjoint element $h\in A$ which is invertible.
We denote these elements by $A_{inv}^{s.a.}$.
\end{definition}
The reasoning behing the definition should be clear: whenever we have a  non-degenerate representation $\pi$ of $A$ on some Hilbert space then $\pi(h)$ is a self adjoint operator whose spectrum does not contain $0$ and thus has a gap at energy $0$. Assuming that the Fermi energy is at $0$ (we work in the one-particle approximation and add a constant to move the Fermi energy to $0$) 
$\pi(h)$ would describe an insulator at low temperature.

To motivate the \CA ic approach to topological phases of insulators let us consider a situation in which we ignore all symmetries.
We say that two abstract insulators $h_1$ and $h_2$ belong to the same topological phase if they are homotopic in $A_{inv}^{s.a.}$. This then implies that in each representation $\pi_\omega$ the Hamiltonians $\pi_\omega(h_1)$ and $\pi_\omega(h_2)$ are homotopic, the homotopies preserving the covariance relation. 
For the ease of formulation we drop the adjective abstract and speak about insulators when we mean elements of $A_{inv}^{s.a.}$.
 
Our "classifying space" is thus $A_{inv}^{s.a.}/\simh$. This is just a set.
In order to say something more about it it is turned into an abelian group. 
There is a formal way to do this: 
Consider $M_n(A)$ in place of $A$ with homotopy in $M_n(A)_{inv}^{s.a.}$. Then the disjoint union
$$ S(A) := \bigsqcup_n (M_n(A)_{inv}^{s.a.}/\simh)$$
is a semigroup under the operation $[x]+[y] = [x\oplus y]$. An application of the Whitehead lemma shows that addition is abelian. Apply the Grothendieck functor to this semigroup to get a group $GS(A)$.
This is a universal abelian group which can be obtained as the quotient $S(A)\times S(A)/\sim$ where $([x_1],[y_1])\sim ([x_2],[y_2])$ whenever there exists $[z]\in S(A)$ such that 
 $[x_1]+[y_2]+[z] = [x_2]+[y_1]+[z]$. Now define $d:S(A) \to \NM$ by $d([x]) = n$ if $x\in M_n(A)$. This is an additive map and by the functorial properties of the Grothendieck construction it induces a group homomorphism $d:GS(A)\to \Z$. Van Daele's $K_0$-group of $A$ (in Roe's reformulation \cite{Roe}) is the subgroup 
 $$\ker d = \{[[x],[y]]\in GS(A)| \exists n\in\NM:x,y\in M_n(A)\} . $$
 It is isomorphic to $KU_0(A)$ the standard $K_0$-group of the ungraded complex \CA\ $A$.
We thus see how $KU_0(A)$ arrises directly from a standard construction in algebraic topology applied   to the classifying space of topological phases.

Note that we need a pair of insulators to define an element of $\ker d$.
By choosing the second insulator to be a trivial one we can assign to a single insulator an element of the group. Indeed, let $e\in A_{inv}^{s.a.}$ be such a choice. Abstractly, $e$ can be any element of $A_{inv}^{s.a.}$, but physically we think of $e$ as being a trivial insulator, that is, one which has strictly positive spectrum. This means that the entire spectrum of $e$ lies above the Fermi-energy and hence no state is occupied.
Now define an equivalence relation $\simhe$ on
 $\bigsqcup_n M_n(A)_{inv}^{s.a.}$ as follows. Two elements $x\in M_k(A)_{inv}^{s.a.}$, $y\in M_l(A)_{inv}^{s.a.}$ are equivalent (we write $x\simhe y$) if 
$x\oplus {e\oplus \cdots \oplus e}$ and $y\oplus e\oplus \cdots \oplus e$ are homotopic
in $M_{m}(A)_{inv}^{s.a.}$ for some large enough $m$. Then $[x]_e+[y]_e = [x\oplus y]_e$ (we denote here equivalence classes for $\simhe $ by $[\cdot]_e$)
defines an abelian semi-group structure on
$$S_e(A) :=\left. \left(\bigsqcup_n M_n(A)_{inv}^{s.a.}\right) \right/ \simhe $$
and the map $\ker d\ni [[x],[y]]\mapsto [[x]_e,[y]_e]$ defines an isomorphism between $\ker d$ and the Grothendieck group $GS_e(A)$ of $S_e(A)$. Furthermore the map $h\mapsto [[h],[e]]$ assigns to an insulator $h$ an element of $GS_e(A)$. 
This is van Daele's version of van Daele's $K_0$-group of $A$ which depends on a choice of "basepoint" $e$.

To be precise, the enlargement of the space and the equivalence relation lead to a different notion of phase, which we call extended topological phase. 
We say that {\em two abstract topological insulators are in the same extended topological phase provided they define the same element in $GS_e(A)$, the $K_0$-group of $A$}. If $S_e(A)$ has the cancelation property, that is, $[x]_e+[z]_e = [y]_e+[z]_e$ implies $[x]_e= [y]_e$, then 
two abstract topological insulators are in the same extended topological phase if and only if
after adding on trivial insulators they are homotopic (the homotopy preserving selfadjointness and invertibility). 

We should mention that 
there exist insulators which are in the same extended topological phase without being in the same topological phase.   
Kennedy and Zirnbauer have obtained a Bott periodicity result for unextended phases \cite{Kennedy}.

One more observation before we proceed with the detailed definitions.
In a \CA\ $A$ any invertible self adjoint element $h$ is homotopic to a self-adjoint unitary, namely its sign $\mathrm{sgn}(h)$;
this is referred to as spectral flattening in the literature. As a result, we may replace in the definition of the group of extended topological phases
self-adjoint invertibles by self-adjoint unitaries.
\section{Graded \CA s and real structures}
The symmetries of topological insulators can mathematically most effectively be described using real structures and gradings. 
\subsection{Graded \CA s}
A short discussion on graded \CA s can be found in \cite{Bla}.
\begin{definition}
A grading on a (complex or real) \CA\ $A$ is  a $*$-automorphism $\gamma$ of order two.
\end{definition} 
A graded \CA\ $(A,\gamma)$ is a \CA\ $A$ equipped with a grading $\gamma$. $\gamma$-invariant elements are called even and elements with $\gamma(a) = -a$ are called odd.  
An element which is either even or odd is called homogeneous.
A grading is called trivial (and the algebra ungraded) if $\gamma=\id$ in which case there are only even elements. 
  
\subsubsection{Clifford algebras}  
Most important for the present subject are the (real and complex) Clifford algebras.
$Cl_{r,s}$ is the graded real \CA\ generated by 
$r$ self-adjoint generators $e_1,\cdots,e_r$ which square to $1$, $e_k^2=1$,
and $s$ anti-self adjoint generators $f_1,\cdots,f_s$ which square to $-1$ and all generators
anticommute pairwise. 
The grading is defined by declaring the generators to be odd. 
The complex Clifford algebras are the complexification of $Cl_{r,s}$, but due to the possibility of multplying the generators with $i$ the distinction between $r$ and $s$ becomes irrelevant and so we denote $\CM\otimes_\R Cl_{r,s}$ also by $\C l_{r+s}$.
We denote all gradings on Clifford algebras simply by $\st$ as the context is always clear. 

We have $\C l_1\cong \C\oplus\C$ with grading given by exchange of the summands
$\fp(a,b) = (b,a)$. Similarily $Cl_{1,0}\cong \R\oplus\R$ with grading given by $\fp$, the odd generator being given by $(1,-1)$. Finally $Cl_{0,1}\cong \{\C\oplus\C\ni (a,\bar a)\}$ with grading given by $\fp$, the odd generator being given by $(i,-i)$.

We will make frequent use of the Pauli matrices
$$\sigma_x = \left(\begin{array}{cc} 0 & 1\\1&0\end{array}\right)\qquad
\sigma_y= \left(\begin{array}{cc} 0 & -i\\i&0\end{array}\right)\qquad
\sigma_z= \left(\begin{array}{cc} 1&0\\0&-1\end{array}\right)$$
which are always considered as self-adjoint elements.
They
allow for convenient descriptions of $\C l_2$ and $Cl_{r,s}$ with $r+s=2$.
Indeed, $Cl_{2,0}$ is generated by $\sigma_x,\sigma_y$, 
$Cl_{1,1}$ by $\sigma_x,i\sigma_y$, and $Cl_{0,2}$ by $i\sigma_x,i\sigma_y$. Furthermore $\C l_2$  is generated by $\sigma_x,\sigma_y$ as a complex Clifford algebra.
In particular $(\C l_{2},\st)\cong (M_2(\C),\Ad_{\sigma_z})$ and  
$(Cl_{1,1},\st)\cong (M_2(\R),\Ad_{\sigma_z})$. 

\subsubsection{Tensor products}
We work with real and  complex \CA s. $A\otimes B$ shall therefore mean the tensor product over $\C$ in case both algebras are complex, but the tensor product over $\R$ if one of the algebras is real. As a consequence, if $A$ is a complex but $B$ a real \CA\ then $A\otimes B$ is the tensor product of $A$ with the complexification of $B$. In particular $A\otimes Cl_{r,s} = A\otimes \C l_{r+s}$ provided $A$ is a complex algebra. 

The algebra $M_n(A)$ of $n\times n$ matrices with entries in $A$ can be identified with
$A\otimes M_n(\R)$, as this is simply   $A\otimes M_n(\C)$ provided $A$ is complex.
The standard way to extend a grading $\gamma$ from $A$ to $M_n(A)$ is to apply it entry wise, i.e.\ as $\gamma\otimes \id$ on $A\otimes M_n(\R)$. We denote this grading by $\gamma_n$ or simply also by $\gamma$ if the context is clear.

The graded tensor product between two graded algebras will be written $\hat\otimes$.
It can be understood as the ordinary tensor product as far as the linear structure is concerned, but multiplication takes care of the Koszul sign rule: 
$$(a_1\hat\otimes b_1) (a_2\hat\otimes b_2) = (-1)^{|a_2||b_1|} (a_1a_2\hat\otimes b_1b_2)$$
where $|a|$ denotes the degree of the (homogeneous) element $a$.
Furthermore the standard choice for the grading on the graded tensor product 
is the product grading. 
A readily checked important example is the following:
\begin{equation}\label{eq-Cl}
(Cl_{r,s}\hat\otimes Cl_{r',s'},\st\otimes\st)\cong (Cl_{r+r',s+s'},\st).
\end{equation}
Note that $(B_1\hat\otimes B_2,\gamma_1\otimes \gamma_2)$ is isomorphic as a graded algebra to $(B_1\otimes B_2,\gamma_1\otimes \gamma_2)$ if one of the gradings is trivial.

Recall that the quaternions $\HM$ form a real \CA\ which is spanned as a vector space by 
$\{1,i\sigma_x,i\sigma_y,i\sigma_z\}$. We consider $\HM$ always as trivially graded. 

\begin{lemma}\label{lem-H-times-Cl}
We have
\begin{enumerate}
\item $(\HM\otimes Cl_{1,0},\id\otimes\st) \cong (Cl_{0,3},\st)$.
\item $(\HM\otimes Cl_{0,1},\id\otimes\st) \cong (Cl_{3,0},\st)$.
\item $(\HM\otimes Cl_{1,1},\id\otimes\st) \cong (Cl_{0,4},\st)$.
\item $(\HM\otimes Cl_{2,0},\id\otimes\st) \cong (Cl_{1,3},\st)$.
\item $(\HM\otimes Cl_{0,2},\id\otimes\st) \cong (Cl_{3,1},\st)$.
\item $(Cl_{1,1}\otimes Cl_{r,s},\st\otimes\st) \cong (Cl_{r+1,s+1},\st)$.
\end{enumerate}
\end{lemma}
\begin{proof} We provide in each case a set of odd generators. The signs of the squares of these generators determine the resulting Clifford algebra.
(1) 
$\{i\sigma_x\otimes e, i\sigma_y\otimes e, i\sigma_z\otimes e\}$ where $e^2=1$.
(2) 
$\{i\sigma_x\otimes f, i\sigma_y\otimes f, i\sigma_z\otimes f\}$ where $f^2=-1$.
(3) 
$\{1\otimes i\sigma_y,i\sigma_x\otimes \sigma_x, i\sigma_y\otimes \sigma_x, i\sigma_z\otimes \sigma_x\}$. 
(4) 
 $\{1\otimes \sigma_x,i\sigma_x\otimes \sigma_y, i\sigma_y\otimes \sigma_y, i\sigma_z\otimes \sigma_y\}$.
(5) $\{1\otimes i\sigma_x,\sigma_x\otimes \sigma_y, \sigma_y\otimes \sigma_y, \sigma_z\otimes \sigma_y\}$.
(11) is Eq.(\ref{eq-M2R}) which will be proven furtherdown.
\end{proof}

\subsubsection{Balanced and inner gradings}
A useful description of the multiplier algebra $\Mm(A)$ is the following. 
If $\pi$ is a faithful non-degenerate representation of $A$ on some Hilbert space then the multiplier algebra of $A$ is (isomorphic to) the set of operators $T$ such that $Ta\in\pi(A)$ and $aT\in \pi(A)$ for all $a\in A$. Any automorphism and any anti-automorphism extends from $A$ to $\Mm(A)$. In particular any grading of $A$ extends to a grading on $\Mm(A)$. The minimal unitization of $A$ is isomorphic to the \CA\ generated by $\pi(A)$ and the identity operator on the Hilbert space. 

\begin{definition}
A grading on a \CA\ is called {\em balanced} if its multiplier algebra contains an odd self-adjoint unitary $e$.

A grading $\gamma$ is called {\em inner} if $\gamma = \Ad_\Gamma$ for some self-adjoint unitary 
$\Gamma$ in the multiplier algebra of $A$. The self-adjoint unitary 
$\Gamma$ is called the generator of $\gamma$ or the grading operator.
\end{definition}
In the literature and in particular also in \cite{Bla} an inner grading is referred to as an even grading. But since even / oddness is an abundant property for insulators we prefer to use the more intuitive word inner.

The grading $\Ad_{\sigma_z}$ on $M_2(\C)$ or on $M_2(\R)$ is inner; it is referred to as the standard even grading and we denote it also by $\stev$. 

Given $(A,\gamma)$, besides the standard extension of $\gamma$ there is another important grading on $M_2(A)=A\otimes M_2(\R)$, namely $\gamma_{ev}:=\gamma\otimes\stev$. Applied to $\gamma=\id$, that is, for trivially graded $A$, one obtains an inner grading on $M_2(A)$ which is called standard even grading on $M_2(A)$.

\begin{lemma}[\cite{Bla}] If $\gamma_1$ is an inner grading then
$(B_1\hat\otimes B_2,\gamma_1\otimes \gamma_2)$ is isomorphic as a graded algebra to $(B_1\otimes B_2,\gamma_1\otimes \gamma_2)$.
\end{lemma}
\begin{proof}
If $\Gamma$ is a generator for $\gamma$ then $b_1\hat\otimes b_2 \mapsto b_1 \Gamma^{|b_2|}\otimes b_2$ defines a graded isomorphism between  $(B_1\hat\otimes B_2,\gamma_1\otimes \gamma_2)$ and $(B_1\otimes B_2,\gamma_1\otimes \gamma_2)$.
\end{proof}
A simple application is the following:
\begin{eqnarray}\label{eq-M2R}
(M_2(\RM)\otimes Cl_{r,s},\Ad_{\sigma_z}\otimes\st) &\cong &(Cl_{r+1,s+1},\st)\\
\label{eq-M2C}
(M_2(\CM)\otimes \C l_{n},\Ad_{\sigma_z}\otimes\st) &\cong& (Cl_{n+2},\st).
\end{eqnarray}

If $\gamma$ is an inner grading we set
$\Pi_\pm = \frac{\Gamma\pm1}2$ and $A_{\epsilon,\epsilon'} = \Pi_\epsilon A\Pi_{\epsilon'}$. 

\begin{prop} \label{lem-Morita}
Let $(A,\gamma)$ be a graded \CA.
\begin{enumerate}
\item \cite{Daele} If $\gamma$ is balanced then
$(M_2(A),\gamma_2)$, $(M_2(A),\gamma_{ev})$ and $(A\hat\otimes Cl_{1,1},\gamma\otimes \st)$ 
are all isomorphic graded \CA s. 
If $A$ is unital then $(M_2(A),\gamma_2)$ contains $(Cl_{2,0},\st)$ as a subalgebra.
\item \cite{Bla}
If $\gamma$ is an inner grading then 
$$(M_2(A),\gamma_{ev})\cong  (A\otimes C l_{1,1},\id\otimes \st).$$
\item 
If 
$\gamma$ is balanced and inner then
$(M_2(A),\gamma_2)\cong (A\hat\otimes Cl_{1,1},\id\otimes \st)$ and moreover
$$(A,\gamma) 
\cong (A_{++}\otimes C l_{1,1},\id\otimes \st).$$\end{enumerate}
\end{prop}
\begin{proof} 
(1) Suppose that $\Mm(A)$ contains an odd self adjoint unitary $e$. Define $\Psi_e: M_2(A)\to M_2(A)$
$$\Psi_e \left( \begin{array}{cc}
a & b\\
c &  d
\end{array}\right) = \left( \begin{array}{cc}
a & be\\
ec &  ede
\end{array}\right).
$$
This map is easily seen a \CA\ isomorphism (it is its own inverse) and such that $\Psi_e\circ \gamma = \gamma_{ev}\circ \Psi_e$.
It follows that $(M_2(A),\gamma_2)$ is isomorphic to $(M_2(A),\gamma_{ev})$.
Since $\gamma_{ev}$ is an inner grading,
$(M_2(A),\gamma_{ev}) = (A\otimes M_2(\R),\gamma\otimes\stev)$ is isomorphic as graded algebra to $(A\hat\otimes Cl_{1,1},\gamma\otimes\st)$. 
Now if $A$ is unital and hence $e$ belongs to $A$ then
$e\hat\otimes 1$ and $1\hat\otimes x$, $x\in Cl_{1,1}$ generate a copy of $Cl_{2,1}$ which contains $Cl_{2,0}$ as a graded subalgebra.

(2) Let $\gamma = \Ad_\Gamma$ with some self adjoint unitary $\Gamma$. Then 
$(M_2(A),\gamma_{ev}) = (A\otimes M_2(\R),\Ad_{\Gamma\otimes \sigma_z})$. Let
\begin{equation}\label{eq-U}
U 
= \frac12 \big( (1-\Gamma)\otimes \sigma_x + (1+\Gamma)\otimes 1)\big)
\end{equation} 
One verifies easily that $U (\Gamma\otimes \sigma_z) U^* = 1 \otimes \sigma_z$ and hence $\Ad_U$ induces a isomorphism between $(A\otimes M_2(\R),\Ad_{\Gamma\otimes \sigma_z})$ and $(A\otimes M_2(\R),\id\otimes \Ad_{\sigma_z})$.

(3) 
The spectral decomposition of $\Gamma$ can be employed to decompose the elements of $A$ into block form:
$$A = \left( \begin{array}{cc}
A_{++} & A_{+-}\\
A_{-+} &  A_{--}
\end{array}\right).
$$
If $e\in\Mm(A)$ is an odd self adjoint unitary of then $e\Pi_+ = \Pi_-e$ and $e\Pi_- = \Pi_+e$.
This allows to define $\psi_e: \left( \begin{array}{cc}
A_{++} & A_{+-}\\
A_{-+} &  A_{--}
\end{array}\right)\to M_2(A_{++})$
$$\psi_e \left( \begin{array}{cc}
a & b\\
c &  d
\end{array}\right) = \left( \begin{array}{cc}
a & be\\
ec &  ede
\end{array}\right)
$$
which provides a graded isomorphism between $(A,\gamma)$ and $(M_2(A_{++}),\stev)\cong (A_{++}\otimes Cl_{1,1},\id\otimes \st)$. 
\end{proof}
\begin{cor}\label{cor-R-times-Clifford}
We have 
\begin{enumerate}
\item $(M_2(\CM)\otimes \C l_{n},\id\otimes\st) \cong (\C l_{n+2},\st)$, if $n\geq 1$.
\item $(M_2(\RM)\otimes Cl_{r,s},\id\otimes\st) \cong (Cl_{r+1,s+1},\st)$ if $r+s\geq 1$.
\end{enumerate}
\end{cor}
\begin{proof} Apply Prop.~\ref{lem-Morita} to $(M_2(\C),\Ad_{\sigma_z}) = (\C l_2,\st)$
and to $(M_2(\R),\Ad_{\sigma_z}) = (C l_{1,1},\st)$ and then use (\ref{eq-M2C}), (\ref{eq-M2R}). (If $r=0$ then $Cl_{0,s}$ contains only an odd element $e$ of square $e^2=-1$, but 
we can still argue as in Prop.~\ref{lem-Morita} if we take $\Psi_e$ to be conjugation with $1\oplus e$.)
\end{proof}
\subsection{Real structures on graded \CA s}
\begin{definition}
A real structure on a complex graded \CA\ $(A,\gamma)$ is  an anti-linear $*$-automorphism $\rs$ of order two which commutes with the grading $\gamma$
(and preserves the norm). 
\end{definition} 
A transposition $\tau:A\to A$ is a complex linear map of order two which exchanges the order of product (is an anti-homomorphism). There is a one to one correspondance between real structures and transpositions, namely if $\rs$ is a real structure then $\rs^*:A\to A$ defined by $\rs^*(a) :=\rs(a^*)$ is a transposition.

A complex \CA\ with a real structure $(A,\rs)$ is also referred to as a \RA\footnote{A lot of the literature in the \CA -community is written for \CA s equipped with transpositions and the expression \RA\ is also used for that case.} 
or as a Real \CA\ (with upper case R).
The $\rs$-invariant elements furnish a (graded) real \CA\ which we call the real subalgebra of the \RA.

We write $(A,\gamma,\rs)$ for a graded \RA\ and call the pair $(\gamma,\rs)$ a graded real structure for $A$. A graded real structure is called balanced if the multiplier algebra of $A$ contains an odd self-adjoint unitary $e$ which is $\rs$-invariant.

An example of a real structure is complex conjugation on $\C$ or complex conjugation of the entries of a matrix. We denote these real structures by $\cc$. Note that they depend on the choice of base in which the matrix is represented.
 
Let $(A,\rs)$ and $(A',\rs')$ be two \CA s with real structures. Then $\rs\otimes \rs'$ is a real structure on $A\otimes A'$ and
$$ (A\otimes A')^{\rs\otimes \rs'} = A^{\rs}\otimes A'^{\rs'}.$$
Indeed, we have $A = A^\rs+iA^\rs$ and hence $A\otimes A' = A^{\rs}\otimes A'^{\rs'} + iA^{\rs}\otimes A'^{\rs'}$.

The standard way to extend a real structure $\rs$ from $A$ to $M_n(A)$ is entry wise. 
We write $\rs_n$ or simply $\rs$ for this extension, if there is not danger of confusion.
Upon the identification of $M_n(A)$ with $A\otimes M_n(\R)$ (tensor product over $\R$!) the standard extension becomes $\rs\otimes\id$. But if we identify $M_n(A)$ with $A\otimes M_n(\C)$ (tensor product over $\C$) then  the extension is $\rs\otimes\cc$ where $\cc$ is complex conjugation on $\C$.
If $n=2$ there will be another extensions of $\rs$ to $M_2(A)$ which will play a role, namely
$\rs^\rh$ which is, upon identification of $M_n(A)$ with $A\otimes M_n(\R)$ given by $\rs\otimes \Ad_{i\sigma_y}$.

\subsubsection{Real Clifford algebras}
We have defined the real Clifford algebras above. They may be seen as the real subalgebra of  a graded \RA\ whose algebra is a complex Clifford algebra. This is almost tautological, as $\C l_{r+s} \cong \C\otimes Cl_{r,s}$ and so we may take $\cc\otimes \id$ as the real structure. More useful for us is to start with another realisation of $\C l_n$ and provide the real structures. It will be sufficient to do this for $n=1$ and $n=2$.

Recall that the complex Clifford algebra $(\C l_1,\st)$ can be seen as $(\C\oplus \C,\fp)$ where $\fp(\lambda,\mu) = (\mu,\lambda)$ is the exchange of summands. Indeed $(1,-1)$ is then its odd generator. There are two possible real structures on $\C l_1$ which commute with the grading $\fp$.
\begin{enumerate}
\item $\rc_{1,0} = \cc$, the complex conjugation. $(\C\oplus \C,\fp,\cc)$ is a graded \RA\ whose
real subalgebra is  $(\R\oplus \R,\fp)\cong 
(Cl_{1,0},\st)$, since $(\C\oplus \C)^\cc \cong \R\oplus \R$.
\item $\rc_{0,1}=\fp\circ\cc$, complex conjugation followed by exchange of the summands.
Now $(\C\oplus \C)^{\fp\circ\cc} \cong \{\C\oplus\C\ni (\lambda,\mu):\bar\mu = \lambda\}$
and $(\C\oplus \C)^{\fp\circ\cc},\fp)\cong (Cl_{0,1},\st)$ as the odd elements are purely imaginary and thus $(i,-i)$ the generator (which squares to $-1$).
\end{enumerate}
Recall that the complex Clifford algebra $(\C l_2,\st)$ can be seen as $(M_2(\C),\Ad_{\sigma_z})$.
There are three interesting real structures which commute with the grading $\Ad_{\sigma_z}$. 
\begin{enumerate}
\item $\rc_{1,1}=\cc$. We have $M_2(\C)^\cc=M_2(\R)$. Odd elements are thus real 
multiples of $\sigma_x$ and $i\sigma_y$ and hence
$$(M_2(\C)^\cc,\Ad_{\sigma_z}) \cong(M_2(\R),\Ad_{\sigma_z}) \cong (Cl_{1,1},\st).$$
\item $\rc_{2,0}=\Ad_{\sigma_x}\circ \cc$. Now odd elements are real multiples of 
$\sigma_x$ and $\sigma_y$ and hence
$$(M_2(\C)^{\rc_{2,0}},\Ad_{\sigma_z}) \cong  (Cl_{2,0},\st).$$
We remark that while $M_2(\C)^{\rc_{2,0}}$ is isomorphic to $M_2(\R)$ as an ungraded real algebra, the isomorphism does not commute with $\Ad_{\sigma_z}$ and so the graded versions are not isomorphic as graded real algebras. 
\item $\rc_{0,2}={\rh}:=\Ad_{i\sigma_y}\circ\cc$. 
Now odd elements are real multiples of 
$i\sigma_x$ and $i\sigma_y$ and hence
$M_2(\C)^{\rh}=\HM$ and 
$$(M_2(\C)^\rh,\Ad_{\sigma_z}) \cong(\HM,\Ad_{\sigma_z}) \cong  (Cl_{0,2},\st).$$
\end{enumerate}
As an aside we remark that the choice $\rs = \Ad_{\sigma_z}\circ \cc$ does not yield anything new, as
$(M_2(\C),\Ad_{\sigma_z},\Ad_{\sigma_z}\circ \cc)$ is isomorphic to $(M_2(\C),\Ad_{\sigma_z},\cc)$ (as graded \RA s). The isomorphism is given by conjugation with the element $\left( \begin{array}{cc}
1 & 0\\
0 & i
\end{array}\right)$. 

\subsubsection{Real and imaginary inner gradings}
Consider an inner graded algebra $(A,\Ad_\Gamma)$ together with a real structure $\rs$ which commutes with $\Ad_\Gamma$. The latter is the case if $\rs(\Gamma) a \rs(\Gamma) = \Gamma a\Gamma$ for all $a\in A$ (in the non-unital case we use the unique extension of $\rs$ to the multiplier algebra). This is satisfied if and only if $\rs(\Gamma) = z \Gamma$ for some $z$ in the center of the multiplier algebra of $A$. Since $\rs(\Gamma)^2=1$ the element must satisfy $z^2=1$ and hence is constant equal to $+1$ or $-1$ provided the Gelfand spectrum of $\Zz(A)$ is connected. 
We distinguish the two cases which arrise if $z$ is a multiple of the identity.
\begin{definition}
A grading $\gamma$ on a \RA\ $(A,\rs)$ is called balanced if $\Mm(A)$ contains an $\rs$-invariant odd self-adjoint unitary. It is called
{\em real inner} (or {\rm imaginary inner}) if $\gamma$ is inner and the grading operator $\Gamma$ satisfies $\rs(\Gamma) = \Gamma$ (or $\rs(\Gamma) = -\Gamma$).
\end{definition}
$(\C l_2,\st)$ is inner graded, as $\st = \Ad_{\sigma_z}$, $\sigma_z\in \C l_2$, and $\sigma_z^2=1$. Of the above discussed real subalgebras, $(C l_{r,s},\st)$, $0\leq r,s$, $r+s=2$, only $C l_{1,1}$ is are inner graded. The \RA\ $(\C l_2,\st,\rc_{1,1})$ is real inner graded and the \RA s $(\C l_2,\st,\rc_{2,0})$ and $(\C l_2,\st,\rc_{0,2})$ are imaginary inner graded.

\begin{lemma}\label{lem-inner-grad}
Let $(B_1,\gamma_1,\rs_1)$ and $(B_2,\gamma_2,\rs_2)$ be two inner graded \RA s. 
If $\gamma_1$ is inner with grading operator $\Gamma$ then the isomorphism 
$b_1\hat\otimes b_2\mapsto b_1\Gamma^{|b_2|}\otimes b_2$ intertwines $\rs_1\otimes \rs_2$ with $\rs_1\otimes \rs_2$  provided the grading is real, and with $\rs_1\otimes \rs_2\circ\gamma_2$ provided the grading is imaginary.
\end{lemma}
\begin{proof}
If $\rs_1(\Gamma) =\Gamma$ then $\rs_1\otimes\rs_2(b_1\Gamma^{|b_2|}\otimes b_2) = \rs_1(b_1)\Gamma^{|b_2|}\otimes\rs_2(b_2)$ from which the statement follows. $\rs_1(\Gamma) =-\Gamma$ then $\rs_1\otimes\rs_2\circ\gamma_2(b_1\Gamma^{|b_2|}\otimes b_2) = (-1)^{|b_2|}\rs_1(b_1)\Gamma^{|b_2|}\otimes(-1)^{|b_2|}\rs_2(b_2) = \rs_1(b_1)\Gamma^{|b_2|}\otimes\rs_2(b_2)$. 
\end{proof}

Note that if $\gamma$ is real inner then $\rs$ preserves the decomposition of $A$ with the help of the spectral projections of $\Gamma$, and moreover $\gamma$
induces an inner grading on $A^\tau$. We denote the restriction of $\rs$ to $A_{++}$ by $\rs_{++}$. If, however, $\gamma$ is imaginary inner then $\rs$ maps $A_{++}$ to $A_{--}$ and $A_{+-}$ to $A_{-+}$. 

 \begin{theorem} \label{lem-Morita-real}
Let $(A,\gamma,\rs)$ be a graded \RA.
\begin{enumerate}
\item If $(\gamma,\rs)$ is balanced then
$(M_2(A),\gamma_2,\rs_2)$, $(M_2(A),\gamma_{ev},\rs_2)$ and $(A\hat\otimes \C l_2,\gamma\otimes \stev,\rs\otimes\rc_{1,1})$
are isomorphic as graded \RA s.
In $A$ is unital their real subalgebra contains $(Cl_{2,0},\st)$ as a graded subalgebra.
\item[(2 +)] If $\gamma$ is a real inner grading then
$(M_2(A),\gamma_{ev},\rs_2)\cong  (A \otimes \C l_{2},\id\otimes \st, \rs\otimes\rc_{1,1})$.
\item[(2 --)] If $\gamma$ is an imaginary inner grading for $(A,\rs)$ then 
$(M_2(A),\gamma_{ev},\rs_2)\cong  (A \otimes \C l_{2},\id\otimes \st, \rs\otimes\rc_{2,0})$.
\item[(3 +)] If $\gamma$ is a balanded real inner grading
then 
$$(M_2(A),\gamma_2,\rs_2)\cong  (A \otimes \C l_{2},\id\otimes \st, \rs\otimes \rc_{1,1}).$$
Moreover
$(A,\gamma,\rs)\cong 
(A_{++}\otimes \C l_2,\id\otimes \st,\rs_{++}\otimes \rc_{1,1})$
so that in particular, $(A^\rs,\gamma)\cong (A_{++}^\rs \otimes C l_{1,1},\id\otimes \st)$.
Furthermore $(A^\rs,\id)\cong (M_2(A_{++}^\rs),\id)$.
\item[(3 --)] If $\gamma$ is a balanded imaginary inner grading for $(A,\rs)$ then 
$$(M_2(A),\gamma_2,\rs_2)\cong  (A \otimes \C l_{2},\id\otimes \st, \rs\otimes\rc_{2,0}).$$
Moreover $\gamma$ is a real inner grading for $\Ad_e\circ\rs$, where $e$ is  an $\rs$-invariant odd selfadjoint unitary, and 
$(A,\gamma,\rs)\cong 
(A_{++}\otimes \C l_2,\id\otimes \st,(\Ad_e\circ\rs)_{++}\otimes \rc_{2,0})$.
In particular
$(A^\rs,\gamma)\cong (A^{\Ad_e\circ\rs}_{++}\otimes C l_{2,0},\id\otimes \st)$. 
Finally $(A^\rs,\id)\cong (M_2(A_{++}^{\Ad_e\circ\rs}),\id)$.
\end{enumerate}
\end{theorem}


\begin{proof} Some elements of this proof are very similar to that of Prop.~\ref{lem-Morita} to which we refer for the following notions.

(1) The same map $\Psi_e: M_2(A)\to M_2(A)$ commutes with the real structure (as $\rs(e) = e$). This allows to conclude as before.
 If follows in particular that $(M_2(A)^\rs,\gamma)$ is isomorphic to 
$(A^\rs\hat\otimes Cl_{1,1},\gamma\otimes\st)$. Now the last claim follows as before from the observation that $e\hat\otimes 1$ and $1\hat\otimes x$, $x\in Cl_{1,1}$ belong to $A^\rs\hat\otimes Cl_{1,1}$ if $A$ is unital.

(2+)  The unitary $U$ from (\ref{eq-U}) is $\rs_2$-invariant an hence provides the isomorphism between the \RA s. 

(2--) The unitary $U$ from (\ref{eq-U}) 
intertwines $\rs_2$ with $\Ad_U\circ \rs_2 \circ \Ad_{U^*} = \Ad_{U\rs_2(U^*)}\circ\rs_2$. Now it is easily seen that $U\rs_2(U^*) = 1\otimes \sigma_x$. Note that $\Ad_{1\otimes \sigma_x}\circ\rs_2 = \rs\otimes \rc_{2,0}$. 
Hence $\Ad_U$ provides an isomorphism between
$(A \otimes \C l_{2},\Ad_\Gamma \otimes \st, \rs\otimes\cc)$ and 
$(A \otimes \C l_{2},\Ad_\id \otimes \st, \rs\otimes\rc_{2,0})$ from which follows also the statement about the real subalgebras.

(3+) The first statement is a direct consequenc of (1) and (2+). 

The map  $\psi_e: \left( \begin{array}{cc}
A_{++} & A_{+-}\\
A_{-+} &  A_{--}
\end{array}\right)\to M_2(A_{++})$  intertwines $\rs$ with the real structure $\psi_e\circ \rs \circ\psi_e^{-1}$ on $M_2(A_{++})$. We have
$$\psi_e\circ \rs  \left( \begin{array}{cc}
a & b\\
c &  d
\end{array}\right) 
= \left( \begin{array}{cc}
\rs(a) & \rs(b)e\\
e\rs(c) &  e\rs(d)e
\end{array}\right)
=\rs_{++,2} \left( \begin{array}{cc}
a & be\\
ec &  ede
\end{array}\right) 
$$
so that
$\psi_e\circ \rs \circ\psi_e^{-1}=\rs_{++,2}$. 
Thus $\psi_e$ induces an isomorphism between $(A,\gamma,\rs)$ and
$(M_2(A_{++}),\id_{ev},{\rs_{++,2}})\cong  (A_{++} \otimes \C l_{2},\id\otimes \st,\rs\otimes\rc_{1,1})$.
Clearly $\psi_e$ induces also an isomorphism on the algebras with trivial grading, and hence
$(M_2(A_{++}^{\rs}),\id)\cong  (A_{++}^{\rs} \otimes Cl_{1,1},\id\otimes \id) \cong (A^\rs,\id)$.

(3--) The first statement is a direct consequenc of (1) and (2-). 

Since $e$ is odd we have $\rs(\Gamma) = -\Gamma = e\Gamma e$ which implies the next statement. Now the  above calculation for the case (3+) shows that $\psi_e\circ \Ad_e\circ \rs \circ\psi_e^{-1}=(\Ad_e\circ \rs)_{++,2}$.
Note that $\psi_e(e) = \sigma_x$ and hence $\psi_e\circ \Ad_e \circ\psi_e^{-1} = \Ad_{\sigma_x}$. Hence 
$\psi_e\circ \rs \circ\psi_e^{-1} = \Ad_{\sigma_x}\circ (\Ad_e\circ \rs)_{++,2}$. 
It follows that
\begin{eqnarray*}
(A,\gamma,\rs) &\cong&
(M_2(A_{++}),\id_{ev},{ \Ad_{\sigma_x}\circ (\Ad_e\circ \rs)_{++,2}}) \\
&\cong &(A_{++} \otimes M_2(\C),\id\otimes \stev,(\Ad_e\circ\rs)_{++}\otimes\rc_{2,0}) 
\end{eqnarray*}
Finally, as in (3+) $\psi_e$ induces an isomorphism between the trivially graded algebras
$(M_2(A_{++}^{\Ad_e\circ\rs}),\id)\cong (A^{\Ad_e\circ\rs},\id)$, and conjugation with the unitary $\frac1{\sqrt{2}}(\Gamma+ie\Gamma)$ yields an isomorphism between $(A^{\Ad_e\circ\rs},\id)$ and $(A^{\rs},\id)$.
\end{proof}
\section{Comparing real structures on graded algebras}
We would like to determine the possible real structures on a graded algebra $(A,\gamma)$ which commute with the grading.
The natural notion of isomorphism between graded \RA s is conjugacy, that is, the existence of a $*$-isomorphism which intertwines the grading and the real structure. Keeping the algebra and the grading fixed one would therefore regard two real structures $\rs$ and $\rs'$ as equivalent if there is a (graded) isomorphism $\alpha$ on $A$ such that $\rs' = \alpha\circ\rs\circ\alpha^{-1}$.  A weaker form of equivalence is when this $*$-isomorphism is inner; in this case we call the two real structures inner conjugate. We would like to determine the possible real structures on a given graded \CA\ $(A,\gamma)$  up to inner conjugation. This seems in general too difficult and so we consider here only the following simpler problems:
Given two real structures $\rs$, $\rs'$ which are inner related in the sense that $\rs'\circ\rs$ is an inner automorphism, 
when are they inner conjugate? Moreover, fix a real structure $\rf$ on $(A,\gamma)$ which serves as a reference. How many real structures are there up to inner conjugacy which are inner related to $\rf$?

\subsection{Inner related real structures}
 \begin{definition}
Let $A$ be a graded complex \CA\ with two real structures  $\rs,\ss$. We call them {\em inner related} if there exists a unitary $u$ in the multiplier algebra of $A$ such that $\ss\circ\rs=\Ad_u$. 
\end{definition}
We call a unitary $u$ such that $\rs'\circ\rs=\Ad_u$ a generator for $\rs'\circ\rs$ or the pair 
$\rs,\rs'$. It is determined up to left multiplication with a unitary element of the center $\Zz(A)$ of the multiplier algebra of $A$. We recall that real structures on graded algebras must commute with the grading.

Below we will use frequently a simple consequence of Gelfand theory for normal elements of commutative \CA s:
If 
$x,y\in B$ are two commuting normal elements with finite 
spectrum in a unital \CA\ $B$
then $xy$ has also finite spectrum. Indeed, since $x$ and $y$ commute we may consider them as elements in a unital commutative subalgebra. By Gelfand theory we may thus view them as continuous functions on some locally compact Hausdorff space. The finiteness of their spectrum is then equivalent to these functions taking values in a finite  set. Products of such functions take their values therefore also in a finite  set. Note that a continuous function on a locally compact space taking values in a finite set is the same as a locally constant function. 

Similarily one shows for a normal element $x$ that if $x^n$ for some natural number $n>0$ has finite spectrum then also $x$ has finite  spectrum.

\newcommand{\alp}{\xi}

By an anti-homomorphism on an algebra we mean a linear map satisfying $\alp(ab) = \alp(b)\alp(a)$. 
If $\rs$ is a real structure then $r^*:A\to A$, $\rs^*(a) := r(a^*)$ is a transposition,
i.e.\ a complex linear $*$-anti-homomorphism of order two.

We also need the following criterion for taking square roots.
\begin{lemma}\label{lem-spectrum}
Let $u$ be a unitary in a unital \CA\ $B$. If its spectrum is not fully $S^1$ then it admits a unitary square root in $B$.
Suppose furthermore that $\alp:B\to B$ is a $\C$-linear $*$-homomorphism or $*$-anti-homomorphism. If $u$ is $\alp$-invariant the square root is also $\alp$-invariant.
\end{lemma}
\begin{proof}
There is a complex number $z$ of modulus $1$ such that $zu$ does not contain $-1$ in its spectrum. Then the principal logarithm $w=\log(zu)$ is well defined. It follows that $v = z^{-\frac12}\exp(\frac12 w)$ is unitary and satisfies $v^2 = u$.  

Suppose now that $\alp(u)=u$.
Since $\alp(zu)=z\alp(u) = zu$ also $w$ is $\alp$-invariant (by analyticity the invariance relation extends from 
the disk of convergence of the power series expansion to the domain of the principal log).
It follows that $v = z^{-\frac12}\exp(\frac12 w)$ is $\alp$-invariant.  
\end{proof}

Given a complex- or anti-linear $*$-homo- or $*$-anti-homomorphism $\alp$, it will be useful to set $$\alp^*(a) = \alp(a^*)$$
and consider the function $\varphi_\alp:U(\Mm(A))\to U(\Mm(A))$ defined on the unitary group of the multiplier algebra of $A$ by
$$\varphi_\alp (u) = u\alp(u).$$ 

\begin{lemma}\label{lem-well}
$\varphi_\alp$ is invariant under left multiplication with elements from $\Zz(A)$ if and only if $\alp^*$ acts trivially on $\Zz(A)$.
\end{lemma}
\begin{proof}
Let $\lambda\in\Zz(A)$. Clearly $\alp$ preserves the center. Hence
$\varphi_\alp(\lambda u)= \alp(\lambda) \lambda \alp(u) u$. Now $\alp(\lambda) \lambda = 1$ if and only if $\alp^*(\lambda^*) = \lambda^*$.
\end{proof}
The following lemma generalises a result from \cite{Harpe} on von Neumann factors. \begin{lemma}\label{lem-Harpe} Let $A$ be a  \CA\ and $\alp$ be a $*$-homomorphism of order two which may be complex- or anti-linear (so $\xi$ is a grading or a real structure). Let $u\in U(\Mm(A))$ and ${\alp'} =  \Ad_u\circ\alp$. We have:
\begin{enumerate}
\item \label{lem-state1}
${\alp'}$ and $\alp$ commute if and only if $\varphi_{\alp^*}(u)\in \Zz(A)$. 
\item \label{lem-state6}   If $\varphi_{\alp^*}(u)\in \Zz(A)$
and $\alp$ acts trivially on $\Zz(A)$ 
then $\varphi_{\alp^*}(u)^2=1$.
\item  \label{lem-state3}
$\xi'$ has order two if and only if $\varphi_\alp(u)$ lies in $\Zz(A)$.
In this case we have furthermore:
\begin{enumerate}
\item $\alp(u)$ and $u$ commute and $\varphi_\alp(u)$ 
is $\alp$-invariant. 
\item  $\varphi_\alp(u)=\varphi_{\alp'}(u)$.
\item ${\alp'}$ and $\alp$ commute if and only if $u^2\in \Zz(A)$.
\item \label{lem-state5}  If $\alp^*$ acts trivially on $\Zz(A)$ then $\varphi_\alp(u)^2=1$.
\end{enumerate}
\end{enumerate}
\end{lemma}
\begin{proof} 
(1) ${\alp'}$ and $\alp$ commute if and only if $\Ad_u(a) = \alp({\alp'}(a))=\alp(u)a \alp(u^*)$ for all $a$. This is the case precisely if $\alp(u^*)u \in \Zz(A)$.

(2) If $\alp$ acts trivially on the center and $\varphi_{\alp^*}(u)$ belongs to it
then $\varphi_{\alp^*}(u) = \alp(\varphi_{\alp^*}(u))=\alp(u)\alp^2(u^*) = \varphi_{\alp^*}(u)^*$.
Thus $\varphi_{\alp^*}(u)$ is a self-adjoint unitary element. Such elements have square $1$.

(3)
We have $\alp\circ\Ad_u = \Ad_{\alp(u)}\circ\alp$. Hence ${\alp'}^2=\Ad_{u\xi(u)}\circ\xi^2$ which is the identity if and only if $\Ad_{u\alp(u)}=\id$, and this is the case if and only if $u \alp(u) \in \Zz(A)$. (a) Since $u \alp(u) \in \Zz(A)$ we have $\alp(u)u = u^*\varphi_\alp(u) u = \varphi_\alp(u)$. 
Hence $\alp(u)$ and $u$ commute and $\varphi_\alp(u)$ is $\alp$-invariant. 
(b) Since ${\alp'}(u) u = u \alp(u)$ we see that 
$\varphi_\alp(u)=\varphi_{\alp'}(u)$.
(c) Since $u^2=\varphi_\alp(u)\varphi_{\alp^*}(u)$ and $\varphi_\alp(u)\in\Zz(A)$
the statement follows from (\ref{lem-state1}).
(d) If $\alp^*$ acts trivially on the center then $\varphi_\alp(u) = \alp^*(\varphi_\alp(u))=(\alp(\varphi_\alp(u)))^* = \varphi_\alp(u)^*$.
\end{proof}
This lemma applies to gradings and real structures but in different ways. 
For instance, the condition Statement~\ref{lem-state5} can never hold for gradings, because $\gamma^*$ is complex conjugation on the center, and similarily, the condition of Statement~\ref{lem-state6} cannot hold for real structures.

Given two real structures on a graded \CA\ with grading $\gamma$ we call a generator $u$ locally homogeneous if $\varphi_{\gamma^*}(u)\in\Zz(A)$ and
$\varphi_{\gamma^*}(u)^2=1$
(if $\Xx(A)$ is connected this means that $u$ is homogeneous).  
\begin{lemma}\label{lem-loc}
Let $\gamma$ be a grading which acts trivially on $\Zz(A)$ and $u\in U(\Mm(A))$. Then 
$\Ad_u\circ\gamma = \gamma\circ \Ad_u$ if and only if $u$ is locally homogeneous.
\end{lemma}	
\begin{proof}
If $u$ is locally homogenuous then $\Ad_u$ is an even operator on $A$. Hence the equation of the lemma is true when evaluated on even and when evaluated on odd elements of $A$, hence true in general. For the converse,
$\Ad_u\circ\gamma = \gamma\circ \Ad_{u}$ implies that  $\Ad_u\circ\gamma$ commutes with 
$\gamma$ and hence Lemma~\ref{lem-Harpe}~(\ref{lem-state1}),(\ref{lem-state6}) 
imply that $\varphi_{\gamma^*}(u)^2=1$.
\end{proof}
The last lemma implies that a generator $u$ for $\rs'\circ\rs$, two real structures on $(A,\gamma)$, must be locally homogeneous, provided $\gamma$ acts trivially on the center. Note that this is the case for all inner gradings.

Let us summarize the conditions appearing in Lemma~\ref{lem-spectrum}, Lemma~\ref{lem-Harpe}, and Lemma~\ref{lem-loc}
(with a slight strenghening of the first): 
\begin{itemize}
\item[(A1)] $\rs^*$ and $\gamma$ act trivially on the center $\Zz(A)$ of the multiplier algebra of $A$.
\item[(A2)] There exists a generator for $\ss\circ\rs$ with two opposite holes in its spectrum, more precisely, there is $z\in S^1$ such that $z$ and $-z$ lie in the complement of the spectrum of the generator.
\end{itemize}
The condition that $\rs^*$ and $\gamma$ act trivially on the center of the algebra is crucial for the independance of $\varphi_\rs(u)$ and $\varphi_{\gamma^*}(u)$ on the choice of the generator. It is clearly satisfied if the center is trivial which is, as we argued, to be expected from a system with disorder. 
However, since most of the literature on topological insulators takes the point of view that the crystalline case is a sufficient approximation to study the topological effects we wish to include this case in our discussion. 
In the crystalline case we face the following problem: The center is not trivial, as it is isomorphic to the algebra of continuous functions over the Brillouin zone, and moreover, the standard time reversal symmetry is a real structure which flips the sign of the quasi-momentum and therefore its corresponding transposition 
does not act trivially on the center of the algebra. To treat such cases we consider alternatively the following assumptions. 
\begin{itemize}
\item[B1] $\rs^*$ (and hence also $\ss^*$) and $\gamma$ preserve the connected components of the Gelfand spectrum $\Xx(A)$ of the center $\Zz(A)$ of multiplier algebra $\Mm(A)$.
\item[B2] There exists a locally homogeneous generator for $\ss\circ \rs$ with finite spectrum. 
\end{itemize}
Note that B1 is trivially satisfied if $\Xx(A)$ is connected.
\begin{lemma}\label{lem-inv}
Let $A$ be a \CA\ with an anti-linear $*$-homomorphism $\alp$ of order two which preserves the connected components of $\Xx(A)$.  Then $\alp^*(\lambda) = \lambda$ for each $\lambda\in\Zz(A)$ with finite spectrum.
\end{lemma}
\begin{proof}
An element $\lambda\in\Zz(A)$ with finite spectrum can be seen as a locally constant (continuous) function on $\Xx(A)$. 
Since $\Zz(A)$ is abelian, 
$\alp^*$ acts on $\Zz(A)$ as a complex linear automorphism of order two and hence corresponds to a homeomorphism of order two on $\Xx(A)$. Since it preserves connected components it acts trivially on functions which are constant on the connected components of $\Xx(A)$.
\end{proof}
\begin{lemma}\label{lem-Harpe-discrete}
Let $(A,\gamma)$ be a graded \CA\ with two inner related real structures $\ss$ and $\rs$. Assume B1 and B2. 
Then $\varphi_\rs(u)$ and $\varphi_{\gamma^*}(u)$ are 
independent of the choice of generator $u$ for $\rs\circ\ss$ (as long as it has finite spectrum). 
Moreover, they are self-adjoint and hence have spectrum contained  in $\{1,-1\}$. 
\end{lemma}
\begin{proof} Let $\alp$ be $\rs$ or $\gamma^*$.
Suppose that $u$ and $v$ are two generators with finite spectrum. Then $\lambda = v^{-1} u\in \Zz(A)$. In particular, $u$ and $v$ commute. It follows that the spectrum of $\lambda$ must be finite. By Lemma~\ref{lem-inv} $\alp^*(\lambda)=\lambda$ which implies 
$\varphi_\alp(\lambda) = 1$ and hence  
$\varphi_\alp(u) = \varphi_\alp(v)$. 
The spectrum of $u$ is finite if and only the spectrum of $\alp(u)$ is finite. By Lemma~\ref{lem-Harpe} $u$ and $\alp(u)$ commute, hence $\varphi_\alp(u)$ has finite spectrum and belongs to the center. Again by Lemma~\ref{lem-inv} we have
$\alp^*( \varphi_\alp(u))=\varphi_\alp(u)$. Since $\varphi_\alp(u)$ is $\alp$-invariant this means that $\varphi_\alp(u)$ is self-adjoint.  
\end{proof}
\begin{definition}
Let $(A,\gamma)$ be a graded \CA\ with two inner related real structures $\ss,\rs$.
Suppose that assumption A1 or assumptions B1, B2 hold.
The relative signs between $\rs$ and $\ss$ are
$$\eta_{\rs,\ss}=(\eta^1_{\rs,\ss},\eta^2_{\rs,\ss}):=(\varphi_\rs(u),\varphi_{\gamma^*}(u))$$ 
where, under assumption A1, $u$ is any generator for $\ss\circ\rs$ or, if B1 and B2 are satisfied, any generator with finite spectrum. 
\end{definition}
Indeed, in the first case Lemma~\ref{lem-well} justifies that the sign is well-defined, whereas under hypothesis B1 and B2 this is follows from Lemma~\ref{lem-Harpe-discrete}.

Note that the sign can be understood as a $\Z_2\times\Z_2$-valued function associating to each connected component of the Gelfand spectrum two signs.
Below we will consider the case in which the sign is a constant multiple of the unit, thus  $+1$ or $-1$. These are, of course, the only options if $\Xx(A)$ is connected.

\subsubsection{$\rs$-invariant generators with finite spectrum}
Note that, if we can choose $u$ such that $\rs(u) = u$ then 
$\varphi_\rs(u) = u^{2}$. We shall see that this implies that $u$ has finite spectrum, a property which we are very much interested in. We provide a criterion for the existence of an $\rs$-invariant generator.
\begin{lemma}\label{lem-invariance}
Let  $\rs$ and $\ss$ be inner related real structures on a \CA\ $A$ such that 
$\Zz(A)=\C$ or that B1 and B2 are satisfied. The following are equivalent:
\begin{enumerate}
\item[(1a)] There exists a generator $u$ with $\rs(u) = u$ and finite spectrum.
\item[(1b)] There exists a generator $u$ with $\rs(u) = u$.
\item[(2)] $\rs$ and $\ss$ commute.
\item[(3)]  
there exists a generator $u$ such that $u^2$ admits an even square root in $\Zz(A)$ which has finite spectrum. 
\end{enumerate}
\end{lemma}
\begin{proof}
Clearly $1a\Rightarrow 1b$.


\noindent
$1b\Rightarrow 2$: If $\rs(u) = u$ then $\rs \circ\Ad_u = \Ad_{\rs(u)} \circ\rs= \Ad_{u} \circ\rs$ showing that $\rs$ and $\ss$ commute. 

\noindent
$2\Rightarrow 3$: Recall that if $\rs$ and $\ss$ commute then $u^2\in\Zz(A)$. 

If $\Zz(A) = \C$ then $u^2\in\C$ and hence admits a square root in $\C$. Clearly the square root has finite spectrum.

If, by B2, the spectrum of $u$ is finite then also $u^2$ has finite spectrum and hence, by Lemma~\ref{lem-spectrum}, admits an even square root in the algebra it generates (which is contained in $\Zz(A)$). This square must also have finite spectrum.

\noindent
$3\Rightarrow 1$: Let $u$ be a generator and $z\in \Zz(A)$ be an even square root of $u^2$ with finite spectrum. Then $v=z^{-1}u$ has finite spectrum and satisfies $v^2=1$. Then $\rs(v) = \rs(v)v^2 = \varphi_\rs(v) v$. 
$\varphi_\rs(v)$ is $\rs^*$-invariant, even and has finite spectrum. It thus admits an $\rs^*$-invariant even square root $\mu\in \Zz(A)$. Let $w=\mu v$.
Then 
$\rs(w) = \mu^*\rs(v) =  \mu^* \mu^2 v = w$.
Clearly the spectrum of $w$ is finite. 
\end{proof}
\subsection{Inner conjugacy}
\begin{definition}
Let $A$ be a graded complex \CA\ with two real structures  $\rs,\ss$. We call them {\em inner conjugate} if there exists a unitary $w\in \Mm(A)$ such that  $ \ss(w b w^*) = w \rs(b) w^*$ and $\Ad_w$ preserves the grading.
\end{definition}
\begin{lemma}\label{lem-ic}
Let  $\ss,\rs$ be real structures on a graded \CA\ $A$. 
They are inner conjugate if and only if there exists a locally homogeneous
unitary $w\in\Mm(A)$ such that 
$$\Ad_{w\rs(w^*)}\circ \rs = \ss .$$
A necessary condition for inner conjugacy is that
the generator $u=w\rs(w^*)$ is $\rs^*$ invariant and locally homogeneous or,
if $\rs^*$ preserves the connected components of $\Xx(A)$, even.
\end{lemma}
\begin{proof} 
Suppose that there exists a unitary $w\in \Mm(A)$ such that 
$ \ss(w b w^*) = w \rs(b) w^*$. Then 
$  \ss(b) = w\rs(w^* b w) w^*= w\rs(w^*) \rs(b) \rs(w)w^* $. The other direction follows similarily. 

We assume now that $u = w\rs(w^*)$.
We have $\rs^*(w\rs(w^*)) = w\rs(w^*)$ and hence $u = w\rs(w^*)$ is $\rs^*$-invariant.
Since $\ss$ preserves $\gamma$, hence $\Ad_w\circ\gamma = \gamma\circ\Ad_w$, we must have $\gamma(w) = \lambda w$ for some $\lambda\in\Zz(A)$. Applying $\gamma$ again we see that $\lambda^2=1$. In particular $\lambda$ has discrete spectrum.
Now $\gamma(w \rs(w^*)) = \lambda \rs(\lambda^*) w\rs(w^*)$ shows that also $\gamma(u) u^*$ lies in the center and has square $1$. Hence $u$ is locally homogeneous. Now if $\rs^*$ preserves the connected components of $\Xx(A)$ we have 
$\rs^*(\lambda)=\lambda$ by Lemma~\ref{lem-inv} and hence $u$ is even.
\end{proof}
\subsubsection{Graded real structures for algebras whose center has connected Gelfand spectrum}
We assume now that $(A,\gamma,\rf)$ is a graded \RA\ such that $\Xx(A)$ is connected.
We wish to classify all real structures $\rs$ which commute with $\gamma$ and 
are inner related to $\rf$. We do this under assumptions A1, A2 or B2 (B1 is trivially satisfied).
We are interested in two cases: a balanced and a trivial grading.

\begin{cor}
Let $(\gamma,\rf)$ be a graded real structure for $A$ and $(\gamma,\rs_i)$, $i=1,2$ be two graded real structures which inner related to $(\gamma,\rf)$. We assume
A1, A2 or B1, B2 for both real structures so that signs are well defined.
If $(\gamma,\rs_1)$ and $(\gamma,\rs_2)$ are inner conjugate then $\eta_{\rs_1,\rf} = \eta_{\rs_2,\rf}$. 
\end{cor}
\begin{proof}
Let $\Theta$ and $u$ be such that $\rs_2 = \Ad_u \circ\rs_1$ and $\rs_1=\Ad_\Theta\circ\rf$. Then $\rs_2(u\Theta)u\Theta = u\rs_1(u)\Theta\rf( \Theta)$. Since $\rs_1$ and $\rs_2$ are inner conjugate, $u\rs_1(u)=1$ by Lemma~\ref{lem-ic} thus the first sign is the same. By the same lemma $u$ is even and hence also the second sign is the same.
\end{proof}

Thus if the signs are different then the real structures $\rs$ and $\rf$ cannot be inner conjugated. The following theorem shows  that (under appropriate conditions on $\Theta$) the signs determine $\rs$ up to inner conjugation and stabilisation, provided $\Xx(A)$ is connected. 
\begin{theorem}\label{thm-sign-graded}
Let $(A,\gamma,\rf)$ be a \CA\ with connected $\Xx(A)$ with either balanced or trivial grading
$\gamma$.
Let $\rs$ be a second  real structure on $(A,\gamma)$ which is inner related to $\rf$ 
and such that
conditions A1, A2 or B2 (B1 is trivially satisfied) are satisfied.
Denote by $\Theta$ a corresponding generator for $\rs\circ\rf$. 
\begin{itemize}
\item[(1)] 
$(M_2(A),\gamma_2,\rs_2)$ is inner conjugate to $(A\hat\otimes M_2(\C),\gamma\otimes\gamma',\rf'\otimes\rs')$ where the grading $\gamma'$ and the real structures $\rf'$ and $\rs'$ depend on $\eta_{\rs,\rf}$ and are listed together with the real subalgebra in the following table. 
\smallskip

\begin{tabular}{|c||c|c|c|c|}
\hline
$\eta_{\rs,\rf}$ & $\gamma'$ &$\rf'$& $\rs'$  & $(M_2(A)^{\rs_2},\gamma_2)$ \\
\hline
\hline
$(+1,+1)$ & $\id$ &$\rf$& $\rc_{1,1}$ & $(A^\rf \hat\otimes M_2(\R),\gamma\otimes\id)$\\
\hline
$(-1,+1)$ & $\id$ &$\rf$& $\rc_{0,2}$ & $(A^\rf \hat\otimes \HM,\gamma\otimes\id)$\\
\hline
$(+1,-1)$ & $\st$ &$\rf\circ\gamma$& $\rc_{2,0}$  & $(A^{\rf\circ\gamma} \hat\otimes C l_{2,0},\gamma\otimes\st)$\\
\hline
$(-1,-1)$ & $\st$ &$\rf\circ\gamma$& $\rc_{0,2}$  & $(A^{\rf\circ\gamma} \hat\otimes C l_{0,2},\gamma\otimes\st)$\\
\hline
\end{tabular}
\smallskip

If $\eta_{\rs,\rf} = (+1,+1)$ then already
$(A,\gamma,\rs)$ is inner conjugate to $(A,\gamma,\rf)$.

\item[(2)]
$(M_2(A)\hat\otimes\C _1,\gamma_2\otimes\st,\rs_2\otimes\ss)$ is inner conjugate to 
$(A\hat\otimes\C l_3,\gamma\otimes\st,\rf'\otimes\rs'')$ where the real structures $\rf'$ and  $\rs''$, which depend
on $\eta_{\rs,\rf}$ and on the real structure $\ss$ on $\C l_1$, are listed together with the real subalgebra in the following table. 

\smallskip

\begin{tabular}{|c|c||c|c|c|}
\hline
$\eta_{\rs,\rf}$ & $\ss$ &$\rf'$&  $\rs''$ & $(M_2(A)\hat\otimes\C l_1)^{\rs_2\otimes\ss},\gamma_2\otimes\st)$ \\
\hline
\hline
$(+1,+1)$ & $\rc_{1,0}$ &$\rf$&$\rc_{2,1}$ & $(A^\rf \hat\otimes C l_{2,1},\gamma\otimes\st)$\\
\hline

$(-1,+1)$ & $\rc_{1,0}$ &$\rf$&$\rc_{0,3,}$ & $(A^\rf \hat\otimes C l_{0,3,},\gamma\otimes\st)$\\
\hline
$(+1,+1)$ & $\rc_{0,1}$ &$\rf$&$\rc_{1,2}$ & $(A^\rf \hat\otimes C l_{1,2},\gamma\otimes\st)$\\
\hline

$(-1,+1)$ & $\rc_{0,1}$ &$\rf$&$\rc_{3,0}$ & $(A^\rf \hat\otimes C l_{3,0},\gamma\otimes\st)$\\
\hline

$(+1,-1)$ & $\rc_{1,0}$ &$\rf\circ\gamma$&  $\rc_{3,0}$ & $(A^{\rf\circ\gamma} \hat\otimes C l_{3,0},\gamma\otimes\st)$\\
\hline

$(-1,-1)$ & $\rc_{1,0}$ &$\rf\circ\gamma$&  $\rc_{1,2}$ & $(A^{\rf\circ\gamma} \hat\otimes C l_{1,2},\gamma\otimes\st)$\\
\hline
$(+1,-1)$ & $\rc_{0,1}$ &$\rf\circ\gamma$&  $\rc_{2,1}$ & $(A^{\rf\circ\gamma} \hat\otimes C l_{2,1},\gamma\otimes\st)$\\
\hline

$(-1,-1)$ & $\rc_{0,1}$ &$\rf\circ\gamma$&  $\rc_{0,3}$ & $(A^{\rf\circ\gamma} \hat\otimes C l_{0,3},\gamma\otimes\st)$\\
\hline
\end{tabular}
\smallskip

If $\eta_{\rs,\rf} = (+1,+1)$ then 
$(A\hat\otimes\C _1,\gamma\otimes\st,\rs\otimes\ss)$ is inner conjugate to $(A\hat\otimes\C _1,\gamma\otimes\st,\rf\otimes\ss)$.
\end{itemize}

\end{theorem}

\begin{proof} Let $\Theta$ be a generator, that is $\rs = \Ad_\Theta\circ\rf$,  which satisfies the assumptions A1, A2 or B2. 
They in particular imply that the signs $\eta_{\rs,\rf}$ are well defined. 

(1)
(i) If $\eta_{\rs,\rf} = (+1,+1)$ then $\rf^*(\Theta) = \Theta$ and $\Theta$ is even.
By assumption its spectrum is not fully $S^1$. 
From Lemma~\ref{lem-spectrum} we conclude that $\Theta$ admits an even $\rf^*$-invariant square root $w$. Hence $\Theta = w^2 = w\rs(w^*)$ and Lemma~\ref{lem-ic} implies that $(A,\gamma,\rs)$ is inner conjugate to $(A,\gamma,\rf)$, which implies that
$(M_2(A),\gamma_2,\rs_2)$ is inner conjugate to 
$(A\otimes M_2(\C),\gamma\otimes\id,\rf\otimes \cc)$.
Since $M_2(\C)$ is trivially graded the ungraded tensor product is here the same as the graded one.

(ii) If $\eta_{\rs,\rf} = (-1,+1)$ then $\rf^*(\Theta) = -\Theta$ and $\Theta$ is even.
Recall that $(M_2(A),\gamma_2,\rs_2)$ can be seen as $(A\otimes M_2(\C),\gamma\otimes \id,\rs\otimes \cc)$. 
As 
$\rc_{0,2}= \Ad_{\sigma_y} \circ \cc$ we may write
$$\Ad_\Theta \circ \rf\otimes \cc 
= \Ad_{\Theta\otimes\sigma_y} \circ (\rf\otimes \rc_{0,2})$$
We have $\rc_{0,2}(\sigma_y) = -\sigma_y$ and $\rf^*(\Theta)=-\Theta$ so that 
$\eta^{(1)}_{\rs\otimes\cc,\rf\otimes\rc_{0,2}}=1$. Furthermore,
the unitary $1\otimes\sigma_y$ is even and hence also 
$\eta^{(2)}_{\rs\otimes\cc,\rf\otimes\rc_{0,2}}=1$
($\Theta\otimes\sigma_y$ is even). Provided the hypotheses B1, B2 or A1, A2 are satisfied for $M_2(A)$ and $\Theta\otimes\sigma_y$ we can
apply case (1) to see that $(A\otimes M_2(\C),\gamma\otimes\id,\rs\otimes\cc)$ is inner conjugate to $(A\otimes M_2(\C),\gamma\otimes\id,\rf\otimes\rc_{0,2})$. 
As above, the ungraded tensor product coincides with the graded one.

It remains to show that  the hypotheses are satisfied for $M_2(A)$ and $\Theta\otimes\sigma_y$:
The multiplier algebras of $M_2(A)$ and $A$ have the same center and $\rs^*$ and $\rs_2^*$ induce the same action on the center. Furthermore, since $\sigma_y$ has spectrum $\{+1,-1\}$ the spectrum of $\Theta\otimes \sigma_y$ is $\spec(\Theta)\cup -\spec(\Theta)$ where $\spec(\Theta)$ denotes the spectrum of $\Theta$. Hence if  $\pm z\notin\spec(\Theta)$ then $\pm z \notin \spec(\Theta\otimes\sigma_y)$.

(iii) If $\eta_{\rs,\rf} = (+1,-1)$ then $\Theta\rs(\Theta) = 1$ and $\Theta$ is odd. 
Hence the grading is non-trivial and hence balanced (by assumption) so that there exists an odd $\rf$-invariant self-adjoint unitary $e\in \Mm(A)$. Then $\Ad_\Theta(e)$ is  an odd $\rs$-invariant self-adjoint unitary and we can apply
Thm.~\ref{lem-Morita-real}(1) to see that $(M_2(A),\gamma_2,\rs_2)$ is inner conjugate to
$(A\hat\otimes \C l_2,\gamma\otimes\st,\rs\otimes\rc_{1,1})$. 
We have
\begin{eqnarray*}
\Ad_\Theta \circ \rf\otimes \rc_{1,1} &=& 
\Ad_\Theta\circ\rf \otimes  \Ad_{\sigma_y} \circ \rc_{0,2}\\
&=&\Ad_{\Theta\hat\otimes\sigma_y} \circ (\rf\circ\gamma\otimes \rc_{0,2}\circ\st)\\
&=& \Ad_{\Theta\hat\otimes\sigma_y} \circ (\rf\circ\gamma\otimes \rc_{2,0})
\end{eqnarray*}
as $\Theta\hat\otimes\sigma_y(\rf\gamma(a)\hat\otimes\rc_{0,2}\st(c))(\Theta^*\hat\otimes\sigma_y)^* = (-1)^{|a|+|c|}\Theta \rf\gamma(a) \Theta^*\hat\otimes \sigma_x\rc_{0,2}\st(c)\sigma_y$. Now
$$\varphi_{\rs\otimes\rc_{1,1}}(\Theta\hat\otimes\sigma_y) = \Theta \hat\otimes\sigma_y\:\rs(\Theta) \hat\otimes \rc_{1,1}(\sigma_y) = -  \Theta\rs(\Theta) \hat\otimes \sigma_y \rc_{1,1}(\sigma_y)=1.$$
Furthermore $1\hat\otimes\sigma_x$ is odd and hence $\Theta\hat\otimes\sigma_x$ is even.
Thus $\eta_{\rs\otimes\rc_{1,1},\rf\circ\gamma\otimes\rc_{2,0}} = (+1,+1)$ and can apply case (1) to see that $(M_2(A),\gamma_2,\rs_2)$ is inner conjugate to
$(A\hat\otimes \C l_2,\gamma\otimes\st,\rf\circ\gamma\otimes\rc_{2,0})$, provided the hypotheses B1, B2 or A1, A2 are satisfied. The latter is shown as in (ii).

(iv) If $\eta_{\rs,\rf} = (-1,-1)$ then $\Theta\rs(\Theta) = -1$ and $\Theta$ is odd. We persue exactly as in case (iii) except for using the identity 
\begin{eqnarray*}
\Ad_\Theta \circ \rf\otimes \rc_{1,1} &=& \Ad_{\Theta\hat\otimes\sigma_x} \circ (\rf\circ\gamma\otimes \rc_{0,2}).
\end{eqnarray*}
This leads to $\eta_{\rs\otimes\rc_{1,1},\rf\circ\gamma\otimes\rc_{0,2}}=(+1,+1)$ and we can apply (1) to see that $(A\hat\otimes \C l_2,\gamma\otimes\st,\rs\otimes\rc_{1,1})$ is inner conjugate to
$(A\hat\otimes M_2(\C),\gamma\otimes\id,\rf\circ\gamma\otimes\rc_{0,2})$.
%


(2) We apply the results from (1) to express $(M_2(A),\gamma_2,\rs_2)$ as a graded tensor product
$(A,\gamma,\rf')\hat\otimes (M_2(\C),\gamma',\rs')$
and then take the graded tensor product with $(\C l_1,\st,\ss)$. We have
$$ (M_2(\C),\gamma',\rs')\hat\otimes (\C l_1,\st,\ss) \cong (\C l_3,\st,\rs'')$$
where $\gamma',\rs'$ depend on $\eta_{\rs,\rf}$ as in the table of (1) and the real structure $\rs''$ depends on the grading $\gamma'$ and the structures $\rs',\ss$. $\rs''$ can be deduced from Lemma~\ref{lem-H-times-Cl} and is listed in the following tables.
\smallskip

\begin{tabular}{|c|c|c||c|}
\hline
$\gamma'$ & $\rs'$ & $\ss$ & $\rs''$  \\
\hline
\hline
$\id$ & $\rc_{1,1}$ & $\rc_{1,0}$ & $\rc_{2,1}$ \\
\hline
$\id$ & $\rc_{0,2}$ & $\rc_{1,0}$ & $\rc_{0,3}$ \\
\hline
$\id$ & $\rc_{1,1}$ & $\rc_{0,1}$ & $\rc_{1,2}$ \\
\hline
$\id$ & $\rc_{0,2}$ & $\rc_{0,1}$ & $\rc_{3,0}$ \\
\hline
\end{tabular}
$\quad $
\begin{tabular}{|c|c|c||c|}
\hline
$\gamma$ & $\rs'$ & $\ss$ & $\rs''$  \\
\hline
\hline
$\st$ & $\rc_{2,0}$ & $\rc_{1,0}$ & $\rc_{3,0}$ \\
\hline

$\st$ & $\rc_{0,2}$ & $\rc_{1,0}$ & $\rc_{1,2}$ \\
\hline
$\st$ & $\rc_{2,0}$ & $\rc_{0,1}$ & $\rc_{2,1}$ \\
\hline
$\st$ & $\rc_{0,2}$ & $\rc_{0,1}$ & $\rc_{0,3}$ \\
\hline
\end{tabular}
\smallskip

This yields the table of (2).
\end{proof}

\begin{cor} \label{cor-sign-graded}
Let $(A,\gamma,\rf)$ be a \CA\ with connected $\Xx(A)$ and balanced or trivial grading.
Up to stabilisation and inner conjugation there are at most four different real structures on $A$ which are inner related to $\rf$ and such that
conditions A1, A2 or B2 are satisfied. 
Likewise up to stabilisation and inner conjugation there are at most four different real structures of the form $\rs\otimes\ss$ on $A\hat\otimes\C l_1$ which are inner related to $\rf\otimes\rc_{1,0}$ or to $\rf\otimes\rc_{0,1}$ and such that for $\rs$ and $\rf$ conditions A1, A2 or B2  are satisfied. 
\end{cor}
\subsubsection{Graded real structures on $M_n(\C)$ and $M_n(\C)\otimes\C l_{1}$}\label{sec-grst}
As an example which is important for insulators we discuss the possible graded real structures 
on $M_n(\C)$ and on
$M_n(\C)\otimes\C l_{1}$ up to stabilisation and conjugation. We are actually only interested in balanced real structures in the first case, and in real structures on $M_n(\C)\otimes\C l_{1}$ for which the grading of the first factor is trivial, because these are the relevant cases for insulators.
\begin{cor} \label{cor-rs}
Up to stabilisation and inner conjugation there are four different balanced graded real structures on $M_n(\C)$ and four different graded real structures of the form $(\id\otimes\phi,\rs\otimes\ss)$ on $M_n(\C)\otimes\C l_{1}$.
\end{cor}
\begin{proof}
All $*$-isomorphisms of $M_n(\C)$ are inner. 
Hence all  gradings are inner. 
Since the grading is balanced the dimension $n$ must be even and the grading operator $\Gamma$ a self-adjoint operator with eigenvalues $+1$ and $-1$ of equal multiplicity. All such operators are inner conjugated. We can therefore choose without loss of generality the grading on $M_n(\C)=M_k(\C)\otimes M_2(\C)$ to be the standard one, namely with grading operator $\Gamma = 1\otimes\sigma_z$. 
Also, all reference structures are inner related. We may therefore take the reference real structure to be anyone which preserves $\Ad_\Gamma$ and choose $\rf=\cc$.
Furthermore assumptions B1, B2 are satisfied, as the center is trivial and all spectra finite anyway. 
We thus get from Cor.~\ref{cor-sign-graded} that there are at most four graded real structures
up to stabilisation and inner conjugation for $M_n(\C)$ and for $M_n(\C)\otimes\C l_1$. 
Table~\ref{tab-rs-graded} and Table~\ref{tab-rs-fp} list each four real structures for $M_n(\C)$ and for $M_n(\C)\otimes\C l_1$, resp., 
realising the four different possiblities for the signs. The given value for $n$ is the smallest possible one. Indeed, in $(M_2(\C),\Ad_{\sigma_z})$ there is no even unitary which satisfies 
$\cc(\Theta)\Theta = -1$ which is why $n>2$ in the fourth case.
\end{proof}
\begin{table}[h]
\begin{center}
\caption{Balanced graded real structures for $M_n(\C)$ and associated real subalgebra.} 
\label{tab-rs-graded}
\begin{tabular}{|c||c||c|c|}
\hline
$n$  & $\rs$ & $\eta_{\rs,\rf}$ & $(M_n(\C)^\rs,\Ad_\Gamma)$ \\
\hline
\hline
$2$  & $\cc$ & $(+1,+1)$ & $Cl_{1,1}$ \\
\hline
$2$ & $\Ad_{\sigma_x}\circ\cc$ &  $(+1,-1)$  & $Cl_{2,0}$ \\
\hline
$2$  & $\Ad_{\sigma_y}\circ\cc$ & $(-1,-1)$ & $Cl_{0,2}$ \\
\hline
$4$  & $\Ad_{\sigma_y}\circ\cc\otimes \cc$ & $(-1,+1)$  & $\HM\otimes Cl_{1,1} \cong Cl_{0,4}$ \\
\hline
\end{tabular}
\end{center}
\end{table}
\begin{table}[h]
\begin{center}
\caption{Graded real structures for $(M_n(\C)\otimes\C l_1,\id\otimes\fp)$  and associated real subalgebra..}
\label{tab-rs-fp}
\begin{tabular}{|c||c||c|c|c|}
\hline
$n$ & $\rs$ &  $\eta^1_{\rs,\rf}$ & $\ss$ & $((M_n(\C)\otimes\C l_1)^\rs,\id\otimes\fp)$ \\
\hline
\hline
$1$ & $\cc$ & $+1$ & $\rc_{1,0} $ & $Cl_{1,0}$ \\
\hline
$1$ & $\fp\circ \cc$ & $+1$ & $\rc_{0,1} $& $Cl_{0,1}$ \\
\hline
$2$ & $\Ad_{\sigma_y}\circ\cc\otimes \cc$ & $-1$ & $\rc_{1,0} $& $\HM\otimes Cl_{1,0} \cong Cl_{0,3}$\\
\hline
$2$ & $\Ad_{\sigma_y}\circ\cc\otimes \fp\circ \cc$ & $-1$ & $\rc_{0,1} $& $\HM\otimes Cl_{0,1} \cong Cl_{3,0}$\\
\hline
\end{tabular}
\end{center}
\end{table}
\subsection{Reference real structure on the observable algebra}\label{sec-ref}
For our application to insulators we wish to define a reference real structure on the observable algebra.

Quite generally, if $A$ is a graded \CA\ represented faithfully on some graded Hilbert space $\Hh$, and if $\cC$ is a complex conjugation on $\Hh$, that is, an antilinear operator of square $1$ which we assume to preserve the grading on $\Hh$, then a real structure $a\mapsto \rf(a)$ may be defined through 
\begin{equation}\label{eq-grf}\pi(\rf(a)) = \cC \pi(a)\cC
\end{equation}
provided $\Ad_\cC$ maps the image of the representation into itself. 
Of course, $\rf$ depends on the representation chosen.
It also depends on the complex conjugation $\cC$. 

If $A$ is inner graded with grading operator $\Gamma$, then it is natural to consider the grading on $\Hh$ defined $\pi(\Gamma)$. It then follows that
$\pi(\rf(\Gamma)) = \cC\pi(\Gamma)\cC$.

We consider first the case in which there is no magnetic field. Let $\Gr=\R^d$ or $\Gr=\Z^d$ (for the tight binding approximation).
Let $\pi = \pi_\omega$ be a representation of the type (\ref{eq-repr}) 
of the observable algebra $A$ (with $\sigma=1$). 
$\Psi\in L^2(\Gr,\C^n)$ is $\C^n$-valued and so if we fix a complex conjugation on $\C^n$ (by choice of a base) we can extend it pointwise to a complex conjugation on $L^2(\Gr,\C^n)$. We denote it by $\Psi\mapsto \overline{\Psi}$. Then the above equation (\ref{eq-grf})becomes
$ \overline{\pi_\omega(F)\Psi} = \pi_\omega(\rf(F))\overline{\Psi}$. 
In terms of $M_n(\C)$-valued integral kernels this means ($x,y\in\Gr$)
\begin{equation}\label{eq-kernel}
\big(\pi_\omega(\rf(F))\big)_{xy} = \overline{\big(\pi_\omega(F)\big)_{xy}}
\end{equation}
where $\overline{M}=\cc(M)$, 
that is entrywise complex conjugation in $M_n(\C)$ 
(the matrices are to be expressend in the above chosen base of $\C^n$). 
We obtain from (\ref{eq-repr})
\begin{equation}\label{eq-rf}
\rf(F) =\overline{F}.
\end{equation}
It is directly verified that $\rf$ is anti-linear, multiplicative, equivariant w.r.t.\ the $*$-operation and of order two. It is hence a real structure on $A$. It will serve us as reference structure.
In the case that $A$ is inner graded with grading operator $\Gamma$ we consider the representation (\ref{eq-repr}) as graded via the grading operator $\pi(\Gamma)$ on $L^2(\Gr,\C^n)$. We then
need to suppose that the complex conjugation on $\C^n$ can be chosen in such a way that $\pi(\Gamma)$ commutes or anti-commutes with the defined complex conjugation on $L^2(\Gr,\C^n)$. In that case $\rf$ commutes or anti-commutes with the grading.

If there is a magnetic field the above cannot just be generalised by incorporating a non-trivial $2$-cocycle $\sigma$ but there are obstructions. This is not too surprising, as it is known that a magnetic field breaks in general time reversal invariance. We explain this obstruction in the case $\Gr=\R^d$.

It is closer to the physical problem to use another representation of $A$, one that is unitarily equivalent to 
(\ref{eq-repr}). 
We make the common assumption that $\Omega$ contains a dense orbit, i.e.\ an element $\omega_0$ such that $\{\alpha_x(\omega_0)|x\in\R^d\}$ is dense. Then 
$C(\Omega)$ is isomorphic to the $C^*$-algebra 
$\Cc$ of continuous functions $\tilde f:\R^d\to\C$ which are of the form $\tilde f(x) =  f(\alpha^{-1}_x(\omega_0))$ for some $f\in C(\Omega)$. As a result, $A$ is isomorphic to 
$$ A \cong \Cc \rtimes_{\tilde\alpha,\tilde\sigma}\R^d $$
where $\tilde\alpha$ is the standard translation action $\tilde\alpha_x(\tilde f)(y) = \tilde f(y+x)$
and $\tilde\sigma:\R^d\times\R^d \to \Cc\cap C(\R^d,S^1)$ is given by $\tilde\sigma(x,y)(q) = \sigma(x,y)(\alpha^{-1}_q(\omega_0))$. Furthermore, under this isomorphism the representation $\rho_{\omega_0}$ and the representation $T_{\omega_0}(a)$ turn into
\begin{eqnarray*}
\rho(\tilde f)\Psi(x) &=& \tilde f(x)\Psi(x) \\
T(a)\Psi(x) & = & \tilde\sigma(x,a)(0)\Psi(x+a).
\end{eqnarray*}
Finally, our condition on the magnetic field now reads $B_{\omega_0}\in\Cc$.

We follow \cite{MPR} to write the $2$-cocycle $\tilde\sigma$ as a {\em pseudo}-co-boundary\footnote{$\tilde\sigma$ is not a genuine co-boundary
of  $\lambda$ as $\lambda$ need not to take values in $\Cc$.}
of a $1$-cochain $\lambda:\R^d\to C(\R^d,S^1)$.  Indeed, let $A$ be a continuous vector potential for the magnetic field $B_{\omega_0}$, that is, a $1$-form satisfying 
$B_{\omega_0} = dA$ ($B_{\omega_0}$ is a $2$-form and $d$ the exterior derivative). Define 
$\lambda: \R^d \to C(\R^d,\SM^1)$ by
$$ \lambda(x)(q) = \exp(-i \int_{[q,q+x]} A).$$  
Then we have $\tilde \sigma = \delta \lambda$ where $\delta$ is the 
coboundary operator for group cohomology, 
$\delta\lambda(a,b) = \lambda(a)\tilde\alpha_b(\lambda(b))\lambda(a+b)^{-1}$.
This allows to define another representation of $A$ on $L^2(\R^d,\C^n)$, namely
\begin{eqnarray*}
\rho(\tilde f)\Psi(x) &=& \tilde f(x)\Psi(x) \\
T^\lambda(a)\Psi(x) & = & \rho(\lambda(a))T(A)\Psi(x) = \lambda(a)(x)\,\Psi(x+a).
\end{eqnarray*}
The transformation $\Psi(x) \mapsto \lambda(x)(0)\,\Psi(x)$ yields a unitary equivalence between this representation and the one above.
We now use the same complex conjugation $\Psi\mapsto \overline{\Psi}$ as above attempting to define a real structure $\rf$ by the analog of Equ.~\ref{eq-kernel} which now reads
$$\rf(\tilde F)(y-x)(x)\:\lambda(y-x)(x) = \overline{\tilde F(y-x)(x)\:\lambda(y-x)(x)}$$
for all $x,y\in\R^d$.
\begin{theorem} Let
$\lambda$ be as above and define for $F\in C_c(\R^d,\Cc)$
$$\rf(F) := \overline{F} \lambda^{-2}$$
where $\overline{F}(z)(x) = \overline{F(z)(x)}$.
$\rf$ is a real structure if and only if for all $z\in \R^d$ we have
$\lambda(z)\in\Cc$, or, equivalently, the components of the vector potential belong to $\Cc$.
\end{theorem}
\begin{proof} Suppose that the above defines a real structure. Since $\Omega$ is compact, $\Cc$ is unital. Thus, for any $z$ there exists $F\in C_c(\R^d,\Cc)$ such that $F(z)$ is a non-vanishing function of $\Cc$. It follows that  
$\lambda^{-2}(z) = \frac{\rf(F)(z)}{\overline{F(z)} }$. 
Hence $\lambda^2(z)\in\Cc$. By definition of $\lambda$, $z\mapsto \lambda(z)(q)$ is 
continuous for all $q$ and since (by assumption)
$\lambda^2(z)$ is uniformly continuous the map $z\mapsto \lambda^2(z)$ is continuous in the sup-norm topology. We can therefore define its logarithm 
via continuous functional calculus, provided $z$ is small enough (so that $-1$ is not in the image of $\lambda^2(z)$). Hence $q\mapsto \int_{[q,q+z]} \frac12 A$ belongs to $\Cc$ for all small $z$ and hence, by additivity, for all $z$. Since $A$ was assumed continuous this implies that the components of $A$ belong to $\Cc$ and then, of course, also 
$\lambda(z)\in \Cc$ for all $z$.

On the other hand, if the condition is satisfied then $\rf$ preserves $C_c(\R^d,\Cc)$
(and extends by continuity to the crossed product) so we only have to show that it is anti-linear, multiplicative, commutes with the $*$-structure and has order two.
Anti-linearity is direct.
$$ \rf(F*G) (x) =\lambda^{-2}(x) \int dy \overline{F(y)\alpha_y(G(x-y))}\sigma^{-1}(y,x-y) $$
$$ \rf(F)*\rf(G) (x) = \int dy \overline{F(y)}\lambda^{-2}(y)\overline{\alpha_y(G(x-y))}
\alpha_y(\lambda^{-2}(x-y))\sigma(y,x-y) $$
These two expressions are equal if
$$\sigma^2(y,x-y) = \lambda^2(y)\alpha_y(\lambda^2(x-y))\lambda^{-2}(x)$$
which is the square of the relation $\sigma = \delta^1 \lambda$.
$$\rf(F^*)(x) = \alpha_x(\overline{F(-x)})^*\lambda^{-2}(x)  =  \alpha_x(F(-x))\lambda^{-2}(x)$$
$$\rf(F)^*(x) =  \alpha_x(\rf(F)(-x))^* = \alpha_x(\overline{\overline{F(-x)}\lambda^{2}(-x)})= \alpha_x(F(-x))\alpha_{x}(\lambda^{2}(-x))$$
These two expressions are equal if
$\lambda^{-2}(x)=\alpha_{x}(\lambda^{2}(-x))$ which follows since
$$\alpha_{x}(\lambda(-x))(q)= \exp(-i\int_{[q+x,q]}A) = \exp(i\int_{[q,q+x]}A) = \lambda^{-1}(x)(q)$$
Finally 
$$\rf^2(F) = \rf(\overline{F}\lambda^{-2})= \overline{\overline{F}\lambda^{-2}}\lambda^{-2} = F$$
\end{proof}
Note that, if $\lambda\in\Cc$ then $\tilde\sigma$ is a proper coboundary in group cohomology with coefficients in $\Cc$. This does not imply that $B_{\omega_0}$ vanishes but that the twisted crossed product $\Cc\rtimes_{\tilde\alpha,\tilde\sigma}\R^d$ is isomorphic to the untwisted one $\Cc\rtimes_{\tilde\alpha,1}\R^d$. We see from the above that if $\Cc$ are periodic functions then a reference structure exists if and only if the vector potential is periodic (with the same lattice of periods). This implies, of course, that the magnetic field is periodic, but not any periodic magnetic field will do it, its flux through a unit cell must be $0$.

\subsection{Graded real structures for tight binding models}
Recall that the observable algebra of a tight binding model for the description of a particle in a solid is given by the crossed product of $C(X,M_n(\C))$ with a $\Z^d$-action $\alpha$ on $X$.
As we have excluded external magnetic fields the reference real structure $\rf(F) = \bar{F}$ from the last section is well defined. 


The results of Section~\ref{sec-grst} allow to obtain real and graded real structures for tight binding models which are inner related to the reference real structure. Indeed the decomposition
$A = M_n(\C)\otimes C(X)\rtimes_\alpha\Z^d$ suggests to apply Cor.~\ref{cor-rs} to the first factor. If $A$ is trivially graded we obtain next to $\rf$ a second real structure $\rs = {\Ad_{1\otimes \sigma_y}}\circ\rf$, provided $M_n(\C) = M_k(\C)\otimes M_2(\C)$. 
The first case will be interpreted as {\em even} and the second as {\em odd} time reversal symmetry. The famous Kane Mele Hamiltonian has the form
$$ h = \begin{pmatrix} h_1 & R \\ R^* & h_2 \end{pmatrix} $$
where $h_1,h_2,R\in M_k(\C)\rtimes_\id\Z^2\cong M_k(C(\TM^2))$ ($k=2$). 
Indeed $h_1$ and $h_2$ are two copies of the Haldane Hamiltonain which are related by 
time reversal $\rf(h_1)=h_2$ and $R$ is a coupling between the two which satisfies $\rf(R)=-R^*$. Note that $\rs(h) = h$ is equivalent to $\rf(h_1)=h_2$, $\rf(R)=-R^*$.

Models with even or odd particle hole symmetry are obtained if one combines 
the above choices for $\rs$ with the real structure $\rc_{0,1}$ on $\C l_1$. We thus obtain four real structures on $A\otimes\C l_1$ similar as in Table~\ref{tab-rs-fp}. 

If we restrict to internal chiral gradings, that is chiral gradings which affect only the internal degrees of freedom thus having the form $\gamma = \gamma'\otimes \id$ in the above decomposition of $A$ then the possibilities are given by Table~\ref{tab-rs-graded}:
We may take $\gamma' = \Ad_{1\otimes \sigma_z}$ on $M_n(\C) = M_k(\C)\otimes M_2(\C)$, and then obtain four real structures, 
$\rs = \rf$, $\rs = \Ad_{1\otimes \sigma_x}\circ\rf$, $\rs = \Ad_{1\otimes \sigma_y}\circ\rf$, and
$\rs = \Ad_{\sigma_y\otimes 1}\circ\rf$ (in the latter case $k$ must be at least $2$). 

\section{Van Daele $K$-theory for \RA s}

Our $K$-theoretical approach to topological insulators is based on van Daele's formulation of $K$-theory \cite{Daele}.  
The particularity of his approach is that he works with graded algebras. 
This point of view comes to full power in the context of insulators, as a chiral symmetry can be interpreted as a grading. 

\subsection{Definition and general results}
Van Daele's formulates his $K$-theory for general graded Banach algebras. 
We first recall his theory for graded \CA s (which simplifies a little bit matters, as on can work with self adjoint unitaries instead of more general invertible elements which square to $1$).
 
Let $(A,\gamma)$ be a unital graded \CA. We denote 
$$\Us(A,\gamma)=\{x\in A : x = x^* = x^{-1}, \gamma(x) = -x\},
$$
$\Us(A,\gamma)$ are the odd self adjoint unitaries of $A$. 
 \begin{definition}[\cite{Daele}]
Let $A$ be a graded unital \CA. Suppose that $\Us(A,\gamma)$ contains an element $e$ which is homotopic to $-e$ in $\Us(A,\gamma)$. The $K$-group of van Daele with basepoint $e$ is
$$\DK_e(A,\gamma) :=\left. \bigsqcup_{n\in \NM} \Us(M_n(A),\gamma_n) \right/ \sim_h^e$$
where $x\sim^e_h y$ if there are $k,l\in\NM$ such that 
$x\oplus e_k$ is homotopic to
$y\oplus e_l$ in $\Us(M_n(A),\gamma_n)$
for sufficiently large $n$. Addition is given by $[x]+[y]=[x\oplus y]$.
\end{definition}
Above we have used the standard extension of $\gamma_n$ of $\gamma$ to $M_n(A)$ and the notation $e_n$ for the direct sum $e\oplus \cdots \oplus e$ of $n$ copies of $e$.  

The incorporation of a real structure $\rs$ can be done in the following way. 
Let $(A,\gamma,\rs)$ be a graded unital \RA.
We extend $\rs$ in the usual way to $M_n(A)=A\otimes M_n(\R)$ as $\rs\otimes \id$:
 $\rs((a_{ij})) = (\rs(a_{ij}))$ and define
$$\Us(A,\gamma,\rs) :=\{x\in A: x^* = x^{-1},\gamma(x)=-x,\rs(x)=x\}.$$
As $\gamma$ and $\rs$ commute $\gamma$ induces a grading on the real \CA\ $A^\rs$ (which we also denote by $\gamma$) and 
$$\Us(A,\gamma,\rs) = \Us(A^\rs,\gamma).$$  
\begin{definition}
Let $(A,\gamma,\rs)$ be a graded unital \RA. Suppose
that $\Us(A,\gamma,\rs)$ contains an element $e$ which is homotopic to $-e$ in 
$\Us(A,\gamma,\rs)$. The $K$-group of van Daele is
$$\DK_e(A,\gamma,\rs) := \DK_e(A^{\rs},\gamma).$$
\end{definition} 
We list the most basic results (see \cite{Daele}).
\begin{enumerate}
\item Addition is abelian, inversion is given by  $[x]\mapsto [-exe]$, and $[e]$ is the neutral element.
\item The map $\Us(A,\gamma)\ni x \mapsto x\oplus e\in \Us(M_2(A),\gamma_2)$  induces an isomorphism $\DK_e(A,\gamma) \cong \DK_{e\oplus e}(M_2(A),\gamma_2)$. We refer to this as stability of $\DK_e$.
\item If $e,f\in \Us(A,\gamma)$ are two self adjoint odd unitaries which are homotopic to their negative in $\Us(A,\gamma)$ then $\DK_{e}A,\gamma)$ and $\DK_{f}(A,\gamma)$ are isomorphic groups.
\item $\DK_e$ is a functor from the category of graded unital \CA s with graded unital $*$-morphisms to the category of abelian groups. Indeed, any such morphism $\varphi:(A,\gamma)\to (B,\gamma')$, extended component wise to matrices, preserves odd self-adjoint unitaries and the equivalence relation and hence induces a group homomorphism $\DK_e(A,\gamma)\to \DK_{\varphi(e)}(B,\gamma')$. 
\end{enumerate}
Note that $(Cl_{2,0},\st)$ is a (real) graded \CA\ which contains an odd self-adjoint unitary which is homotopic to its negative in $\Us(Cl_{2,0},\st)$. Indeed, 
whenever we have two anticommuting odd self-adjoint unitaries--and here we can take
the two generators $e_1$ and $e_2$ of $Cl_{2,0}$--then 
$w(t) = \cos(t) e_1 + \sin(t) e_2$ is a continuous path of odd self-adjoint unitaries.
This shows that $e_1 = w(0)$ is homotopic to $w(\pi) = -e_1$ and, by the way also, 
that $e_1$ is homotopic to $w(\pi/2) = e_2$. This observation allows to extend the definition of 
van Daele's $K$-group to cases in which $A$ contains an odd self adjoint unitary $e$ but perhaps none which is homotopic to its negative. Indeed, if $e$ is an odd self adjoint unitary in $A$ then
$\begin{pmatrix} e & 0 \\ 0 & -e \end{pmatrix}$ is an odd self adjoint unitary in $(M_2(A),\gamma_{ev})$
which anticommutes with $\sigma_x$. The latter is also odd and therefore homotopic to its negative and we may define
\begin{equation}\label{eq-DK}
\DK(A,\gamma) := \DK_{\sigma_x}(M_2(A),\gamma_{ev}).
\end{equation}
This definition extends the one given above, because by Prop.~\ref{lem-Morita-real} $(M_2(A),\gamma_{ev})$ is isomorphic to
$(M_2(A),\gamma_2)$ and hence, if $e$ happens to be homotopic to its negative then 
$\DK_{e_2}(M_2(A),\gamma_2)$ is isomorphic to $\DK_{e_2}(M_2(A),\gamma_{ev})$
which, in turn, is then isomorphic to $\DK(A,\gamma)$. Finally, by stability
$\DK_{e_2}(M_2(A),\gamma_2)$ is isomorphic to $\DK_{e}(A,\gamma)$.

Definition (\ref{eq-DK}) works also for algebras which do not contain any odd self-adjoint unitary.
Indeed, van Daele shows that $\DK_{\sigma_x}(M_2(A),\gamma_{ev})$ 
is always a group which is isomorphic to 
$\DK_{\sigma_z\otimes\sigma_x}(M_2(M_2(A)),{\gamma_{ev}}_{2})$ \cite{Daele3}[Prop.~4 and 5] 
and, by
Prop.~\ref{lem-Morita}, $(M_2(M_2(A)),{\gamma_{ev}}_{2})$ is isomorphic to 
$(M_4(A),\gamma_4)$.\footnote{In \cite{Daele},
$\DK_{\sigma_z\otimes\sigma_x}(M_2(M_2(A)),{\gamma_{ev}}_{2})$ has been taken as the general definition of the $K$-group.}

Note that $\DK(A,\gamma)$ does not depend on any choice for $e$. However, this is can lead to missconception for two reasons. First, the isomorphism between $(M_2(A),\gamma_{ev})$ and  $(M_2(A),\gamma_2)$ depends on a choice of odd self-adjoint unitary, and second,
if $x\in \Us(A,\gamma)$ then its class is not yet defined in $\DK_{\sigma_x}(M_2(A),\gamma_{ev})$, to achieve this we need an inclusion $\Us(A,\gamma)\hookrightarrow \Us(M_2(A),\gamma_2)$ which is not canonical.

Alternatively, one can define van Daele's $K$-group for an algebra $(A,\gamma)$ with a choice of odd self-adjoint unitary $e\in A$  as $GV_e(A,\gamma)$, the 
Grothendieck construction applied to
$$V_e(A,\gamma):=\left. \bigsqcup_{n\in \NM} \Us(M_n(A),\gamma_n) \right/ \sim_h^e$$ 
which is always a semigroup with the addition defined above. Indeed,
\begin{lemma}[\cite{Daele}] \label{lem-KDK} 
Let $(A,\gamma)$ be a balanced graded algebra and $e\in \Us(A,\gamma)$.
The map
$i_e:\Us(A,\gamma)\to \Us(M_2(A),\gamma_{2})$, $i_e(x) = \begin{pmatrix} x & 0 \\ 0 & -e\end{pmatrix}$ induces the 
group isomorphism
$G(i_e): GV_e(A,\gamma) \to \DK_{e\oplus -e}(M_2(A),\gamma_{2})$ given by
$$G(i_e)([[x],[y]]) = \left[ \begin{pmatrix} x & 0 \\ 0 & - e_ny e_n\end{pmatrix}\right]$$
(for $x,y\in \Us(M_n(A),\gamma_n)$).
\end{lemma}
The class defined by an odd selfadjoint unitary $x$ can therefore be seen in $\DK(A,\gamma)$ in two ways. Either as the element $[[x],[e]]\in GV_e(A,\gamma)$ or as $\left[ \begin{pmatrix} x & 0 \\ 0 & - e\end{pmatrix}\right]$ in $\DK_{e\oplus -e}(M_2(A),\gamma_{2})$. 
\begin{cor}\label{cor-KDK} 
Consider two odd selfadjoint unitaries $x\in M_{n}(A)$, $y\in M_{m}(A)$.
If $V_e(A,\gamma)$ has cancelation then they
 define the same class  in $\DK(A,\gamma)$ if and only if there exist $k$ such that $x\oplus e_{k}$ and $y\oplus e_{n+k-m}$ are homotopic in $\Us(M_{n+k}(A),\gamma)$. Otherwise $x$ and $y$  
 define the same class  in $\DK(A,\gamma)$ if and only if there exist $k$ such that $x\oplus e_{k}\oplus -e_{n+k} $ is homotopic to $y\oplus e_{n+k-m} \oplus -e_{n+k}$ in $\Us(M_{2(n+k)}(A),\gamma)$.
\end{cor}
\begin{proof} 
If $V_e(A,\gamma)$ has cancelation then $GV_e(A,\gamma)$ is the quotient of $V_e(A,\gamma)\times V_e(A,\gamma)$ by the equivalence relation $([x],[x']) \sim ([y],[y'])$ iff 
$[x]+[y'] = [y]+[x']$. This implies the first statement. If  $V_e(A,\gamma)$ does not have cancelation then we work with $ \DK_{e\oplus -e}(M_2(A),\gamma_{2})$ which coincides with 
$ V_{e\oplus -e}(M_2(A),\gamma_{2})$ and is exactly defined by the homotopy formulated in the second statement.
\end{proof}

\subsection{Higher $K$-groups}
Van Daele's higher $K$-groups are defined by tensoring the algebra with Clifford algebras.
Stability and Cor.~\ref{cor-R-times-Clifford} imply that 
$$\DK_e(A\hat \otimes C l_{r,s},\gamma\otimes\st)\cong \DK_e(A\hat \otimes C l_{r+1,s+1},\gamma\otimes\st)$$
and hence $\DK_e(A\hat \otimes C l_{r,s},\gamma\otimes\st)$ depends only on the difference $s-r$.
\begin{definition}
Let $(A,\gamma)$ be a unital graded complex or real \CA\ containing an odd self-adjoint unitary $e$. The $K_n$-group of $(A,\gamma)$ in van Daele's formulation is
$$K_{n}(A,\gamma) := \DK_e(A\hat \otimes C l_{1,n},\gamma\otimes\st).$$ 
\end{definition}
Roe shows that $K_{n}(A,\gamma)$ coincides with Kasparov's $KK$-group $KK_{n}(\CM,A)$ (or $KK_{n}(\RM,A)$) for graded algebras \cite{Roe}.   

For \RA s the definition is similar, $ K_{n}(A,\gamma,\rs) = K_{n}(A^\rs,\gamma)$.

Since $Cl_{8,0}\cong M_{16}(\R)$ we have
$K_{n} = K_{n-8}$.
If $A$ is a complex \CA\ then the tensor product has to be understood over the reals, or equivalently $A\hat \otimes C l_{r,s}\cong A\hat\otimes \C l_{r+s}$. It follows that in the complex case one has even $K_{n} = K_{n-2}$.\footnote{This periodicity of the $K$-groups is a direct consequence of the definition of the higher $K$-groups. The far deeper result referred to as Bott-periodicity concerns the isomorphism between $K_{n+1}(A,\gamma)$ and $K_{n}(SA,\gamma)$ where $SA$ is the suspension of $A$.}

\subsection{Non-unital algebras}
The $K_{n}$ are functors yielding group homomorphisms when applied to graded unital $*$-morphisms.  
Whenever one has such a functor from unital \CA s (or \RA s) to abelian groups (which is split exact) there is a standard procedure to extend it to non-unital \CA s and non-unital morphisms. This is based on adding a unit to $A$ and then using the surjection $A^+\to \CM$ (or $A^+\to \RM$) to define $K_{n}(A,\gamma,\rs)$ as the kernel of the induced map on the groups, $K_{n}(\pi):K_{n}(A^+,\gamma^+,\rs^+)\to K_{n}(\C,\id,\cc)$. Here $\gamma^+$ and $\rs^+$ are the standard extensions to $A^+$.
We won't go into the details of that which are, for instance, nicely explained in \cite{Roerdam} for the ungraded case without real structure (but the arguments extend easily).

\subsection{Trivially graded algebras}
A  trivially graded algebra $(A,\id)$ has no odd elements but this does not matter as
$(A\otimes Cl_{r,s},\id\otimes \st)$ 
contains odd self-adjoint unitaries as soon as $r>0$ and so we can define $K_n$ with 
$n=s-r+1$ for sufficiently high $r$ and $s$. 
\subsubsection{Trivially graded complex \CA s}\label{sec-4.4.1}
Van Daele shows \cite{Daele} for a graded complex \CA\ $(A,\id)$ that
$$K_n(A,\id) = \DK(A\otimes \C l_{n+1},\id\otimes \st)\cong KU_n(A)$$
where $KU_n(A)$ is the standard $K_n$-group ($n$ taken modulo $2$) of $A$ seen as ungraded complex \CA. We briefly describe this.

Let  $n=0$ so that $\C l_1 = \C\oplus \C$ is the relevant Clifford algebra. The odd elements of 
$A\otimes \C l_1$ are of the form $(x,-x)$, $x\in A$. Suppose that $A$ is unital. $(x,-x)$ is an odd self adjoint unitary provided $x$ is a self adjoint unitary. In particular, we can always take $x=+1$ or $x=-1$. None of the two choices leads to an odd self-adjoint unitary which is homotopic to its negative, because the spectra of self-ajoint unitary are contained in $\{+1,-1\}$ and hence preserved by homotopy. Therefore
$V_{(1,-1)}(A\otimes\C l_1,\id\otimes\st)$
is only a semigroup and $\DK(A\otimes\C l_1,\id\otimes\st)$ its associated Grothendieck group.
Given a self-adjoint unitary $x$ let  $p:=\frac{1-x}2$ which is the projection onto to the eigenspace to eigenvalue $-1$.  
The map $(x,-x)\mapsto \frac{1-p}2$ induces the isomorphism of semigroups between 
$V_{(1,-1)}(A\otimes\C l_1,\id\otimes\st)$ and $V(A):=  \bigsqcup_{n\in \NM} \mathrm{Proj}(M_n(A)) / \sim_h^{0}$ where the inclusion of projections $\mathrm{Proj}(M_n(A)) 
\hookrightarrow \mathrm{Proj}(M_{n+1}(A))$ is defined by $p\mapsto p\oplus 0$.
Now, in the standard picture of $K_0$-theory for complex algebras, $KU_0(A)$ is the Grothendieck group associated to $V(A)$. Hence
$K_0(A,\id) \cong KU_0(A)$. The non-unital case follows now from the functoriality of adding a unit.

In case $n=1$ we need to consider $\C l_2 = M_2(\C)$ and $(A\otimes \C l_2,\id\otimes\st)$ is isomorphic to $(M_2(A),\id_{ev})$. Again we may assume that $A$ is unital the non-unital case following by functoriality.
$KU_1(A)$ is constructed from unitaries:
$$KU_1(A) = \bigsqcup_{n\in\NM} U(M_n(A))/\sim_h^1$$
with addition induced by direct sum and $x\sim_h^1 y$ whenever $x\oplus 1\oplus\cdots\oplus 1$ is homotopic to $x\oplus 1\oplus\cdots\oplus 1$ in some $U(M_n(A)$ 
with large enough $n$. An odd self-adjoint unitary in $(M_2(A),\id_{ev})$ must be of the form 
 $\left( \begin{array}{cc}
0 & b\\
c &  0
\end{array}\right)$
with $c = b^* = b^{-1}$. Hence $b$ is a unitary in $A$. The map  $\left( \begin{array}{cc}
0 & b\\
b^* &  0
\end{array}\right)\mapsto b$ induces the isomorphism between $K_1(A,\id)$ and $KU_1(A)$.
 
\subsubsection{Trivially graded \RA s}

Let $(A,\id,\rs)$ be a unital \RA\ which is trivially graded. As above we will employ the trick of passing to the algebra $A \otimes \C l_n$ 
turning it into a graded algebra with grading $\gamma=\id\otimes\st$. 
We have to extend $\rs$ to $A \otimes \C l_n$.   
There are many inequivalent ways, namely each real structure on $\C l_n$ which commutes with the grading on $\C l_n$ gives rise to such an extension. Indeed, let $\rc_{r,s}$ be a real structure on $\C l_{r+s}$ such that $\C l_{r+s}^{\rc_{r,s}} = Cl_{r,s}$ then 
\begin{equation}\label{eq-DK-RA}
\DK(A\otimes \C l_{r+s},\id\otimes\st,\rs\otimes \rc_{r,s}) 
\cong K_{s-r+1}(A^\rs,\id)\cong KO_{s-r+1}(A^\rs),
\end{equation}
where $KO_n(A)$ is the $K_n$-group of $A$ as real ungraded \CA.
The second isomorphism can be inferred from the work of Roe \cite{Roe}
via the identification of $KO_n(A)$ with $KK_n(\R,A)$ for real \CA s $A$ (for $r=s+1$ and 
$r=s$ the direct argument of van Daele sketched in Section~\ref{sec-4.4.1} will also apply to this real case).
%
\subsection{Inner graded algebras}
If $(A,\gamma)$ is an inner graded algebra we can say something more about $\DK(A,\gamma) = K_1(A,\gamma)$.

\subsubsection{Inner graded complex \CA s}
\begin{prop}\label{cor-inner}
Let $(A,\gamma)$ be an inner graded complex \CA.
Then 
$$\DK(A,\gamma)\cong K_1(A,\id)\cong KU_1(A).$$
\end{prop}
\begin{proof}
If $\Us(A,\gamma)$ is not empty we apply Prop.~\ref{lem-Morita}~(2), notably
$(M_2(A),\gamma_2)\cong (A\otimes\C l_2,\id\otimes \st)$, 
to obtain the first isomorphism. Otherwise we replace first $(A,\gamma)$ by $(A\otimes \C l_2,\gamma\otimes\st)$, which by stability of the $K$-groups does not change the group, and proceed as above. 
The second isomorphism was already discussed above.
\end{proof}
The inner graded case contains of course the trivially graded as a special case and 
Prop.~\ref{lem-Morita} says that, up to stabilisation the two are the same. But if we have an interest to avoid as much as possible stabilisation, namely to identify homotopy classes of $\Us(A)$ with those defining $KU_1(A)$ then we can do this provided $A$ is balanced.
Indeed, the map $Q_e : \Us(A,\gamma) \to U(A_{++})$ 
\begin{equation}\label{eq-Qe}
 Q_e(h) =  \Pi_+ eh\Pi_+
\end{equation}
induces an isomorphism $\DK_e(A,\gamma)\to KU_1(A_{++})$. Moreover
$KU_1(A_{++})\cong KU_1(A)$ (as follows from  Prop.~\ref{lem-Morita}~(3)).

\subsubsection{Inner graded \RA s}
\begin{prop}\label{cor-inner-real}
Let $(A,\gamma,\rs)$ be a graded \RA.
\begin{enumerate}
\item If $\gamma$ is real even then
$$\DK(A,\gamma,\rs)\cong K_1(A^\rs,\id)\cong KO_1(A^\rs).$$
\item If $\gamma$ is imaginary even then
$$\DK(A,\gamma,\rs)\cong K_{-1}(A^\rs,\id)\cong KO_{-1}(A^\rs).$$
\end{enumerate}
\end{prop}
\begin{proof}
For the same reason as in the last corollary we may assume the existence of an odd self-adjoint unitary $e$. In the first case we apply Thm.~\ref{lem-Morita-real} (2+) which shows that
$(M_2(A),\gamma_2,\rs_2)\cong (A\otimes \C l_{2},\id\otimes \st,\rs\otimes \rc_{1,1})$. The real subalgebra is thus isomorphic to $(A^\rs\otimes \C l_{1,1},\id\otimes \st)$ which
implies, similar to the complex case $\DK(A,\gamma,\rs)\cong K_1(A^\rs,\id)$. In the second case we apply 
Thm.~\ref{lem-Morita-real} (2-) to obtain 
$(M_2(A),\gamma_2,\rs_2)\cong (A\otimes\C l_2,\id\otimes \st,\rs\otimes\rc_{2,0}))$. Now the real subalgebra is isomorphic to $(A^\rs\otimes \C l_{2,0},\id\otimes \st)$ and we get
$\DK(A,\gamma,\rs)\cong K_{-1}(A^\rs,\id)$. The other isomorphisms had been discussed above.
\end{proof}

\subsection{Boersema and Loring's formulation of $KO$-theory}
Recently Boersema and Loring \cite{Loring} have formulated 
$KO$-theory of ungraded unital real \CA s by means of unitaries with certain symmetries.
This description can be obtained from van Daele's description as follows.

The real \CA\ is given as the real subalgebra $A^\rs$ of a \RA\ $(A,\rs)$ and the symmetries are formulated with the help of the real structure $\rs$.\footnote{In \cite{Loring}
the results are formulated in terms of the transposition $\rs^*$.}
Starting point is the Eq.~\ref{eq-DK-RA} for small values of $r,s$, at least one of them being strictly positive. Given $r,s$ we look for the smallest $k$ for which $(M_k(A)^{\rs_k}\otimes C l_{r,s},\id\otimes\st)$ contains an odd self-adjoint element $e$ such that $e\oplus e$ is homotopic in $\Us(M_{2k}(A)^{\rs_{2k}}\otimes Cl_{r,s},{\id\otimes\st})$ to its negative.
In fact, $k$ is $1$ or $2$, as $M_4(\R)\otimes Cl_{r,s}\cong Cl_{r+2,s+2}$. We then equate
$KO_n(A^\rs)$ with $\DK_e(M_k(A)\otimes \C l_{r+s},\id\otimes\st,\rs_k\otimes\rc_{r,s})$ where $s-r=n-1$.

%

We consider first values for $r,s$ such that $r+s=1$. An odd element $x\in (A\otimes\C l_1,\id\otimes\st)$ is of the form
$$ x = u\otimes (1,-1)$$ and the conditions $x=x^*=x^{-1}$ translate into $u^*=u=u^{-1}$. 
There are two cases:
\begin{itemize}
\item The case $s=0$ which corresponds to $KO_0(A^r)$. 
We may take $e= \sigma_z\otimes(1,-1)\in M_2(A)\otimes \C l_1$ as a basepoint, as it
is an
$\rs_2\otimes\rc_{1,0}$-invariant odd self adjoint unitary which is homotopic to its negative.
Indeed, $e$ anticommutes with $ \sigma_z\otimes(1,-1)\in M_2(A)\otimes \C l_1$.
Clearly $\rs\otimes\rc_{1,0}(x) = x$ is equivalent to $\rs(u) = u$.
Thus the elements of $KO_0(A^\rs)$ are represented by self adjoint $\rs$-invariant unitaries $u$ in $M_2(A)$ (or $M_{2m}(A)$).
\item In the case $s=1$ corresponding to $KO_2(A^\rs)$
we take $e= \sigma_y\otimes(1,-1)\in M_2(A)\otimes \C l_1$ as a basepoint,
as it is $\rs_2\otimes\rc_{0,1}$-invariant and $e\oplus e$ homotopic to its negative. Indeed,
 $X:=\begin{pmatrix} 0 & \sigma_x \\ \sigma_x & 0 
\end{pmatrix}\otimes (1,1)$ is homotopic to $1_4\otimes (1,1)$ in the set of unitaries of
$(M_{4}(A)^{\rs_{2k}}\otimes Cl_{0,1},{\id\otimes\st})$ and $X(e\oplus e) X^* = -(e\oplus e)$. 
Furthermore, $\rs\otimes\rc_{0,1}(x) = x$ iff $\rs(u) = -u$.
Thus the elements of $KO_2(A^\rs)$ are represented by self adjoint unitaries $u\in M_2(A)$  (or $M_{2m}(A)$) which satisfy $\rs(u) = -u$.
\end{itemize}

We consider next values for $r,s$ such that $r+s=2$. An odd element $x\in (A\otimes\C l_2,\id\otimes\st)$ is of the form
$$ x = u\otimes \sigma_x+v\otimes \sigma_y = \begin{pmatrix} 0 & U \\ U^* & 0
\end{pmatrix},\quad U = u +iv $$ 
and the conditions $x=x^*=x^{-1}$ translate into 
$U:=u+iv$ being unitary. 
%

\begin{itemize}
\item Let $s=0$ which corresponds to $KO_{-1}(A^\rs)$.
We have $\rc_{2,0}(\sigma_x) = \sigma_x$ and  $\rc_{2,0}(\sigma_y) = \sigma_y$. Then $1\otimes \sigma_x\in A\otimes \C l_2$ is an $\rs\otimes\rc_{2,0}$-invariant odd self adjoint unitary which is homotopic to its negative, as it anticommutes with $1\otimes \sigma_y$. We thus take it as base point.
Furthermore,
$\rs\otimes\rc_{2,0}(x) = x$ is equivalent to $\rs(u) = u$ and $\rs(v) =v$ which can be expressed as $\rs(U) = U^*$.
The elements of $KO_{-1}(A^\rs)$ are thus represented by unitaries $U$ in $A$ (or $M_m(A)$) which satisfy $\rs(U) = U^*$.
\item If $s=1$ which corresponds to $KO_{1}(A^\rs)$ we have $\rc_{1,1}(\sigma_x) = \sigma_x$ and  $\rc_{1,1}(\sigma_y) = -\sigma_y$. We may take $e=1\otimes \sigma_x$ as base point as
$e\oplus e$ anti-commutes with the odd element $\begin{pmatrix} 0 & i\sigma_y \\ -i\sigma_y & 0 
\end{pmatrix}\in M_2(A^\rs)\otimes Cl_{1,1}$ which has square $1$.
Now
$\rs\otimes\rc_{1,1}(x) = x$ is equivalent to $\rs(u) = u$ and $\rs(v) =-v$ which can be expressed as $\rs(U) = U$. This is a result we already discussed above, namely
the elements of $KO_1(A^\rs)$ are represented by $\rs$-invariant unitaries in $A$ (or $M_m(A)$).
\item In the case $s=2$ corresponding to $KO_{3}(A^\rs)$ we have
$\rc_{0,2}(\sigma_x) = -\sigma_x$ and  $\rc_{0,2}(\sigma_y) = -\sigma_y$. 
We may not find an odd self adjoint unitary in $A^\rs\otimes Cl_{0,2}$.  But
$e=\sigma_y\otimes \sigma_y \in  M_2(A)^{\rs_2}\otimes Cl_{0,2}$ is an odd self adjoint unitary
which moreover anti-commutes with $e=\sigma_y\otimes \sigma_x \in  M_2(A)^{\rs_2}\otimes Cl_{0,2}$ and is thus homotopic to its negative.
Now
$\rs\otimes\rc_{2,0}(x) = x$ is equivalent to $\rs(u) = -u$ and $\rs(v) =-v$, that is,
$\rs(U) = - U^*$.
The elements of $KO_{3}(A^\rs)$ are thus represented by unitaries $U\in M_2(A)$ (or $M_{2m}(A)$) which satisfy $\rs(U) = - U^*$.
\end{itemize}
For the remaining degrees we need to take $r+s = 3$ and $r+s=4$ and
make use of the isomorphisms listed in Lemma~\ref{lem-H-times-Cl} (1-3), notably
\begin{eqnarray*}
(A\otimes \C l_3,\id\otimes\st,\rs\otimes \rc_{0,3})&\cong& (A\otimes M_2(\C)\otimes \C l_1,\id\otimes\id\otimes \st,\rs\otimes\rh\otimes \rc_{1,0})\\
(A\otimes \C l_3,\id\otimes\st,\rs\otimes \rc_{3,0})&\cong& (A\otimes M_2(\C)\otimes \C l_1,\id\otimes\id\otimes \st,\rs\otimes\rh\otimes \rc_{0,1})\\
(A\otimes \C l_4,\id\otimes\st,\rs\otimes \rc_{0,4})&\cong& (A\otimes M_2(\C)\otimes \C l_2,\id\otimes\id\otimes \st,\rs\otimes\rh\otimes \rc_{1,1})
\end{eqnarray*}
and apply the above reasoning to unitaries in $M_2(A)=A\otimes M_2(\C)$ with $\rs^\rh=\rs\otimes\rh$ in place of $\rs$. With similar arguments as above we obtain:
\begin{itemize}
\item In the case $r=0, s=3$ we may take $\sigma_z\otimes 1\otimes (1,-1) \in M_2(A)^{\rs_2}\otimes\HM\otimes Cl_{1,0}$ as basepoint and
the elements of $KO_{4}(A^\rs)$ are represented by self-adjoint unitaries $u$ in $M_4(A)$ (or $M_{4m}(A)$) which satisfy $\rs^\rh (u) = u$.
\item If $r=3$ and $s=0$ we can take $1\otimes\sigma_y\otimes (1,-1) \in A^{\rs}\otimes\HM\otimes Cl_{0,1}$ as basepoint. 
The elements of $KO_{6}(A^\rs)$ are represented by self-adjoint unitaries $u$ in $M_2(A)$ (or $M_{2m}(A)$) which satisfy $\rs^\rh (u) = -u$.
\item Finally, for $r=0, s=4$ we may take $1\otimes 1\otimes \sigma_z \in A^{\rs}\otimes\HM\otimes Cl_{1,1}$ as basepoint, and 
the elements of $KO_{5}(A^\rs)$ are represented by unitaries $U$ in $M_2(A)$ (or $M_{2m}(A)$) which satisfy $\rs^\rh (U) = U$.
\end{itemize}
\section{Extended topological phases of abstract insulators}
We now describe the classification of extended topological phases of insulators. 
We have argued in Section~\ref{sec-2.3} that, if no symmetries are present and $A$ is unital, then the extended topological phases are in bijection with the Grothendieck group $GS_e(A)$ of $S_e(A)=\bigsqcup_n M_n(A)_{inv}^{s.a.} / \simhe $. Furthermore, 
by spectral flattening $S_e(A)$ can be identified with 
$ \bigsqcup_n U^{s.a.}_n(A) / \simhe$ where $U^{s.a.}_n(A)$ are the self-adjoint unitaries in $M_n(A)$.
So we only have to bring a grading into the scheme to obtain van Daele's $K$-groups.
This can be done by the trick: consider the graded algebra $(A\otimes \C l_1,\id\otimes\fp)$ then
the map $(x,-x)\mapsto x$ yields a bijection between
$\Us(M_n(A)\otimes \C_1,\id\otimes\st)$ and $U^{s.a.}_n(A)$ which preserves the equivalence relation so that 
$$GS_e(A) \cong \DK_{e\otimes (1,-1)}(A\otimes\C l_1,\id\otimes \fp).$$
Thus the extended topological phases correspond to the elements of the group $K_0(A,\id)\cong KU_0(A)$. 
In particular, an insulator $h$ defines the element $[[h\otimes (1,-1)],[e\otimes (1,-1)]]\in \DK_{e\otimes (1,-1)}(A\otimes\C l_1,\id\otimes \fp)$. Corrolary~\ref{cor-KDK} provides us with a precise condition under which two insulators yield the same element and hence define the same extended topological phase. If $V_e(A,\gamma)$ has the cancelation property then 
two insulators define the same extended topological phase if and only if after adding on
trivial insulators they are homotopic, but if $V_e(A,\gamma)$ does not satisfy the cancelation property then two insulators of the same extended topological phases are in general only homotopic after adding on trivial and negative trivial insulators. A negative trivial insulator is an insulator whose Hamiltonian $h$ has strictly negative  spectrum, and thus all of its states are occupied and it is homotopic to $-1$. 
While a negative trivial insulator will be sensitive to the gap-labelling its topological transport coefficients obtained from pairings with higher Chern characters must vanish.

\subsection{Insulators with symmetry}

\begin{definition} Let $A$ be a \CA\ and $h\in A$ an abstract Hamiltonian, i.e.\ a sefl-adjoint element.
\begin{enumerate}
\item $h$  has chiral symmetry if there is a grading $\gamma$ on $A$ such that $\gamma(h) = -h$. 
\item $h$  has time reversal symmetry (TRS) if there is a real structure $\rt$ on $A$ such
such that  $\rt(h) = h$. 
\item $h$  particle hole symmetry (PHS) if there is a real structure $\rp$ on $A$ such
such that  $\rp(h) = -h$. 
\end{enumerate}
\end{definition}
The reasoning behind this definition is the following: Usually the symmetries of an insulator are implemented on the Hilbert space by unitary or anti-unitary operators whose squares are $\pm1$ and the action on the Hamiltonian (or any operator) is given by conjugation. 
If we think of the algebra as being faithfully represented on the Hilbert space
and the conjugation by such an operator preserves the image of the representation then it defines an automorphism of order $2$ on the algebra, or a real structure in case the operator is anti-linear.

Insulators may have several of these symmetries, but then these are assumed to commute.
An insulator may have other symmetries which are given by (complex) automorphisms which leave the Hamiltonian invariant (of order two or not, like periodicities). These kind of "ordinary" symmetries are treated separately, for instance, by factoring them out or performing a Fourier-Bloch transformation. We do not consider them here. For that reason, two distinct chiral symmetries, two distinct time reversal symmetries, or two distinct particle hole symmetries will not be considered, as their product would be an ordinary symmetry. Note that the product of a time reversal symmetry together with a particle hole symmetry (which commute) yields a chiral symmetry and we may as well describe the same situation by saying that we have a time reversal symmetry and a chiral symmetry (which commute). We thus have the following combinations:
\begin{itemize}
\item no symmetry
\item chiral symmetry
\item time reversal symmetry
\item particle hole symmetry
\item chiral symmetry and time reversal symmetry which commute (and hence also, by taking their product, particle hole symmetry).
\end{itemize}  
Now it is a very suggestive guess that the first case is described by an ungraded complex \CA, the second by a graded complex \CA, the third and the fourth by an ungraded \RA, and the last case by a graded \RA. This is indeed what happens, but with slight twists which bring in the various different $K_n$-groups.

\subsection{Rough classification}
Under minimal assumptions we get the classification of topological phases by \CA s as displayed in Table~\ref{tab-1} which we explain below.

\subsubsection{No symmetry}
We have discussed the case of no symmetry above under the assumption 
that $A$ is unital: The relevant algebra is the trivially graded complex \CA\ $A$ but to apply the unifying framework of van Daele $K$-theory in which insulators correspond to odd self-adjoint elements we employed the trick to extend the algebra to $A\otimes\C l_1$ graded by $\id\otimes \fp$. This leads to the $K$-group
$$\DK(A\otimes\C l_1,\id\otimes \fp) = K_0(A,\id) \cong KU_0(A).$$

A non-unital algebra does not contain any invertible elements and so the above cannot just be applied to non-unital algebras. But we can argue as follows:
Let $h\in A$ be a self-adjoint element in the non-unital \CA\ $A$ whose minimal unitization we denote by $A^+$. Recall that all $K$-groups of $A$ are defined as the kernel of the map induced  on $K$-theory by the projection $\pi:A^+\to \C$ (or $\R$ in the real case).
Suppose that $h$ has a gap at $\mu$. $\mu$ can't be $0$. The spectral projections 
$P_{\leq\mu}(h)$ and $P_{\geq \mu}(h)$ on the spectral part below and above $\mu$ belong to $A^+$, as these projections are continuous bounded functions of $h$. Furthermore,	
$\pi$ applied to a projection in $A^+$ is either $1$ or $0$, in fact it is $0$ iff the projection belongs to $A$. Since $1=\pi(  P_{\leq\mu}(h) + P_{\geq \mu}(h)) = \pi(  P_{\leq\mu}(h)) + \pi(P_{\geq \mu}(h))$ we see that either $P_{\leq\mu}(h)$ or $P_{\geq \mu}(h))$ belongs to $A$. 
We will assume that $P_{\geq \mu}(h))$ belongs to $A$, otherwise we work with $-h$ instead of $h$. Let $\tilde h = \mathrm{sgn}(h-\mu)$, the spectrally flattened Hamiltonian shifted in energy such that the gap is at $0$. Then $\tilde h = P_{\geq\mu}(h) - P_{\leq \mu}(h))$ and $\pi(\tilde h) = +1$. It follows that the element in $K_0(A^+,\gamma)$ defined by $\tilde h\otimes (1,-1)\in A^+\hat\otimes\C l_1$  lies in the kernel induced by $\pi$ and hence in $K_0(A,\gamma)$.

%


\subsubsection{Chiral symmetry}
If the insulator has only chiral symmetry then it is most natural to consider the associated algebra as a graded complex \CA\ $(A,\gamma)$ the grading being given by the chiral symmetry. Indeed, the odd self-adjoint unitaries of $A$ are then precisely the flattened insulators with chiral symmetry. 
The classification of extended topological phases of insulators with chiral symmetry is thus given by the group 
 $$\DK(A,\gamma)=K_1(A,\gamma).$$ 
Again the above argument requires that $A$ is unital, as we worked with invertible elements.
But now we cannot argue for the non-unital case as above, because if $h$ has chiral symmetry, 
$h-\mu$ has no longer chiral symmetry. We will see that this situation can be improved if we have an inner grading.

We should also mention that chiral symmetry is incompatible with strictly positive spectrum
and hence there is no trivial insulator with chiral symmetry. The classification of extended topological phases with chiral symmetry remains therefore relative.
\subsubsection{Time reversal symmetry}
If the insulator has only time reversal symmetry then it is most natural to consider the associated algebra to be a trivially graded \RA\ $(A,\rt)$ the real structure being given by the time reversal symmetry. We overcome the apparent difficulty that this algebra has no odd elements by the same trick as above:
We consider $(A\otimes\C l_1,\id\otimes\fp,\rt\otimes\cc)$ and identify the Hamiltonian $h$ with the odd element $h\otimes (1,-1)\in A\otimes\C l_1$. Since $\rt\otimes\cc( h\otimes (1,-1)) =h\otimes (1,-1)$ we see that this element lies in the real subalgebra 
$(A\otimes\C l_1)^{\rt\otimes\rc_{1,1}}$. Therefore, the classification of extended topological phases of insulators with only time reversal symmetry is given by the group 
 $$\DK(A\otimes\C l_1,\id\otimes\fp,\rt\otimes\rc_{1,1})=K_0(A^\rt,\id)\cong KO_0(A^\rt).$$ 
Again the above direct reasoning needs a unital $A$, but it is easily seen that the non-unital case can be handled in exactly the same way as the case where there is no symmetry at all, 
because if $\rt(h) = h$ then also $\rt(h-\mu) = h-\mu$.
Hence the class of the spectrally flattened shifted Hamiltonian $h-\mu$ belongs to 
$KO_0(A^\rt)$ also in the non-unital case. 
\subsubsection{Partical hole symmetry}
If the insulator has only particle hole symmetry then we consider again a trivially graded \RA\ $(A,\rp)$ which we extend to $(A\otimes\C l_1,\id\otimes\fp)$.
But this time we use another extension of the real structure, namely we choose the extension
$\rp\otimes\rc_{0,1}= \rp\otimes\fp\circ\cc$. Indeed, this extension absorbs the minus sign in the behaviour of the insulator under PHS:
if we identify the Hamiltonian $h$ with the odd element $h\otimes (1,-1)\in A\otimes\C l_1$ then $\rp\otimes\fp\circ\cc( h\otimes (1,-1)) =-h\otimes (-1,1)$ so that $h\otimes (1,-1)$ belongs to the real subalgebra 
$(A\otimes\C l_1)^{\rp\otimes\rc_{0,1}}$. 
Therefore, the relevant algebra is 
$(A\otimes\C l_1,\id\otimes\fp,\rp\otimes\rc_{0,1})$ and the 
classification of extended topological phases of insulators with only particle hole symmetry is given by the group 
 $$\DK(A\otimes\C l_1,\id\otimes\fp,\rp\otimes\rc_{0,1})=
 \DK(A^\rp\otimes C l_{0,1},\id\otimes\fp) =
 K_2(A^\rp,\id)\cong KO_2(A^\rp).$$ 
As in the case of chiral symmetry we can't work here with a non-unital algebra neither is there a trivial insulator with particle hole symmetry, 
as $\rp(h-\mu) = -h-\mu \neq -(h-\mu)$.  
\subsubsection{Chiral symmetry and time reversal symmetry}
If the insulator has both, chiral symmetry $\gamma$ and time reversal symmetry $\rt$ (which commute) then it is most natural to consider $(A,\gamma,\rt)$ as the relevant algebra. Indeed, the odd $\rt$-invariant self-adjoint unitaries of $A$ are then precisely the flattened insulators with chiral symmetry and time reversal symmetry. 
The classification of extended topological phases of such insulators is thus given by the group 
 $$\DK(A,\gamma,\rt)=K_1(A^\rt,\gamma).$$ 
Also in this case we have to assume $A$ to be unital, except, see below, in the case that the chiral symmetry is inner. 
 \begin{table}[h]
\begin{center}
\caption{Rough classification of topological phases.}
\label{tab-1}
\begin{tabular}{|l|l|l|l|}
\hline
Symmetries & graded algebra & real subalgebra & $K$-group\\
\hline
\hline
none & $(A\!\otimes\! \C l_1,\stot)$ & {\small not applicable}& $KU_0(A)$\\
chiral $\gamma$& $(A,\gamma)$ & {\small not applicable}& $K_1(A,\gamma)$\\
\hline
TRS $\tp$ & $(A\!\otimes\! \C l_1,\stot,\tp\!\otimes\! \rc_{1,0})$ & 
$(A^\tp\!\otimes\! C l_{1,0},\stot)$ & $KO_0(A^\rt)$\\
PHS $\rp$ & $(A\!\otimes\! \C l_1,\stot,\rp\!\otimes\! \rc_{0,1})$ & 
$(A^\rp\!\otimes\! C l_{0,1},\stot)$& $KO_2(A^\rp)$\\
chiral, TRS $\gamma,\tp$ & $(A,\gamma,\tp)$ & 
$(A^\tp,\gamma)$& $K_1(A^\rt,\gamma)$\\
\hline
\end{tabular}
\end{center}
\end{table}
\subsection{Classification w.r.t.\ a reference real structure}
A finer classification of insulators with real symmetry arrises if we  take into account the relative signs of the symmetry w.r.t.\ a reference real structure. From the mathematical point of view this is a bit ad hoc and the reference real structure (which we denote by $\rf$) can be any real structure to which the real structure of the symmetry is inner related, but for physical reasons we think of $\rf$ as the complex conjugation defined via the physical representations of $A$,
that is, $\rf$ is given by (\ref{eq-rf}). Recall that this is only possible for vanishing or very special external magnetic fields, something which is not too surprising as magnetic fields tend to break time reversal symmetry.

For the below classification we make the assumptions needed for Thm.~\ref{thm-sign-graded}.
\begin{itemize}
\item[(H1)] The Gelfand spectrum of the center of the multiplier algebra of the observable algebra is connected. Equivalently, this center contains no other central projections then $0$ and $1$.
We argued that this is the case for systems for which the space of configurations $\Omega$ (or $X$) has a dense orbit, which  is the case for crystalline models and
ought to be the case for disordered systems. 
\item[(H2)] There is a reference real structure $\rf$ to which 
time reversal symmetry (or particle hole symmetry, in case there is no time reversal symmetry) is inner related, that is, $\rs = \Ad_\Theta\circ \rf$ for some homogeneous 
unitary $\Theta$ in the multiplier algebra of the observable algebra. 
This unitary has finite spectrum or, if $\rf^*$ acts trivially on the center, the complement of its spectrum in $S^1$ contains two points $\pm z$. 
\end{itemize}
These conditions imply in particular that the signs of the real structures are well defined.

\begin{definition}
A real symmetry $\rs$ is called even (odd) if the relative sign $\eta^{(1)}_{\rs,\rf}$ to the reference structure $\rf$ is $+1$ ($-1$). 
\end{definition}
This definition is justified by the fact that, if $\rf$ is defined via (\ref{eq-grf}) and $\rs=\Ad_\Theta\circ\rf$ then 
$\eta^{(1)}_{\rs,\rf} = (\pi(\Theta)\cC)^2$ which is the usual definition of even and odd symmetries.
The first of the two signs, $\eta^{(1)}_{\rs,\rf}$, has thus the interpretation of parity of the symmetry. We will see that $\eta^{(2)}_{\rs,\rf}$ is of a similar nature, 
provided the grading is inner.

In the case with chiral symmetry we may assume that the graded real structure $(\gamma,\rt)$ (and hence also $(\gamma,\rf)$) is balanced, because otherwise there would be no insulator. We obtain Table~\ref{tab-1r} the results following directly from Thm.~\ref{thm-sign-graded}(1). 

\begin{table}[h]
\begin{center}
\caption{Classification for insulators with chiral and real symmetry relative to a reference structure $\rf$.}
\label{tab-1r}
\begin{tabular}{|c||c|c|}
\hline
$\eta_{\rt,\rf}$  & real subalgebra & $K$-group \\
\hline
\hline
$(+1,+1)$ &  
$(A^\rf,\gamma)$ & $K_1(A^\rf,\gamma)$ \\
\hline
$(+1,-1)$ &  $(A^{\rf\circ\gamma} \hat\otimes C l_{2,0},\gamma\otimes\st)$& $K_{-1}(A^{\rf\circ\gamma},\gamma)$\\
\hline
$(-1,+1)$ &  $(A^\rf \hat\otimes \HM,\gamma\otimes\id)$& $K_5(A^\rf,\gamma)$\\
\hline
$(-1,-1)$ &  $(A^{\rf\circ\gamma} \hat\otimes C l_{0,2},\gamma\otimes\st)$& $K_3(A^{\rf\circ\gamma},\gamma)$\\
\hline
\end{tabular}
\end{center}
\end{table}
\medskip

In the absence of chiral symmetry $A$ is trivially graded and we have to consider $A\otimes\C l_1$ with graded real structure $(\id\otimes\phi,\rs\otimes\ss)$. If follows that $\eta^{(2)}_{\rs,\rf}=+1$. The first sign $\eta^{(1)}_{\rs,\rf}$ is the parity of the symmetry. We summarise the results in Table~\ref{tab-or}, they follow directly from 
Thm.~\ref{thm-sign-graded}(2).

\begin{table}[h]
\begin{center}
\caption{Classification for insulators with one real symmetry relative to a reference structure $\rf$ (no chiral symmetry).}
\label{tab-or}
\begin{tabular}{|l|c|c||c|c|}
\hline
symmetry & $\eta^{(1)}_{\rs,\rf}$ & $\ss$ &  
real subalgebra & $K$-group \\
\hline
\hline
TRS even & $+1$ & $\rc_{1,0}$  & 
$(A^\rf \hat\otimes C l_{1,0},\id\otimes\st)$ & $KO_0(A^\rf)$\\
\hline
TRS odd & $-1$ & $\rc_{1,0}$  & 
$(A^\rf \hat\otimes C l_{0,3,},\id\otimes\st)$& $KO_4(A^\rf)$\\
\hline
PHS even & $+1$ & $\rc_{0,1}$  & 
$(A^\rf \hat\otimes C l_{0,1},\id\otimes\st)$& $KO_2(A^\rf)$\\
\hline
PHS odd & $-1$ & $\rc_{0,1}$  & 
$(A^\rf \hat\otimes C l_{3,0},\id\otimes\st)$& $KO_6(A^\rf)$\\
\hline
\end{tabular}
\end{center}
\end{table}

\subsection{Classification of inner chiral symmetry}
Recall that a chiral symmetry is called inner if it is a symmetry of the form $\gamma=\Ad_\Gamma$ for some self-adjoint unitary $\Gamma$. This is the case usually considered in physics. Indeed, usually chiral symmetry is implemented on the Hilbert space through conjugation with a self-adjoint unitary 
$\Gamma$. If we think of $A$ as being represented non-degenerately and faithfully on the Hilbert space $\Hh$ via $\pi$ then conjugation with a self-adjoint unitary $\Gamma\in \Bb(\Hh)$ induces an order two automorphism  on $A$, provided that the conjugation preserves $\pi(A)$. What we suppose here is that $\Gamma\pi(a)$ and $\pi(a)\Gamma$ belong to $\pi(A)$ for all $a\in A$. Then $\Gamma$ is an element of the multiplier algebra of $A$.

Again we assume that the grading is balanced (otherwise there would not be an insulator), that is, there is an odd selfadjoint unitary $e$. 

The presence of inner chiral symmetry has several consequences. First,  van Daele's $K_i(A,\gamma)$ or $K_i(A^\rt,\gamma)$ can be related to standard $KO$-groups, second 
we can define and work with $A_{++}$ the compression of $A$ onto the positive spectral part of $\Gamma$. This will allow us to 
consider also non-unital $A$.

\subsubsection{Only inner chiral symmetry}
Recall that insulators which have chiral symmetry are described by a graded complex \CA\ $(A,\gamma)$. If the grading is inner then, by
Prop.~\ref{cor-inner}, the classification of extended topological phases of such insulators is given by the group 
 $$K_1(A,\gamma)\cong KU_1(A).$$ 
Furthermore, by Prop.~\ref{lem-Morita}, $(A,\gamma)\cong (A_{++}\otimes \C l_2,\id\otimes \st)$ and 
the grading operator $\Gamma$ allows us to treat the case of non-unital algebras. Suppose that $A$ is non-unital but that $h$ is an invertible element of the algebra $A^\Gamma$ generated by $A$ and $\Gamma$, which is certainly unital. This algebra contains the projections $\Pi_+$ and $\Pi_-$ and assuming that the grading is balanced it is easy to see that the isomorphism $\Psi_e$ from Prop.~\ref{lem-Morita} extends to an isomorphism between $(A^\Gamma,\Ad_\Gamma)$ and
$(M_2(A_{++}^+),\Ad_{\sigma_z})$ where $A_{++}^+$ is the minimal unitization of $A_{++}$. It identifies $h$ with $\Psi_e(h)=\begin{pmatrix} 0 & Q_e(h) \\ Q_e(h)^* & 0 \end{pmatrix}$ where $Q_e(h)=\Pi_+eh\Pi_+$. Let $\tilde h = \mathrm{sgn}(h)$ and denote by $\pi$ the projection $A^+_{++}\to\C$ extended to $2\times 2$-matrices. Then 
$\pi(\Psi_e(\tilde h)) = \begin{pmatrix} 0 & z \\ \bar{z} & 0 \end{pmatrix}$ for some complex number $z$ of modulus $1$. This matrix is homotopic to $\sigma_x$ inside the set of odd self-adjoint unitaries.
If we take $e=\sigma_x$ as base point we see that $[\pi(\Psi_e(\tilde h))]$ represents the trivial class in $\DK_e(\C l_2,\st)$ and hence $[\tilde h]\in \DK_e(A_{++}\otimes \C l_2,\id\otimes \st)\cong KU_1(A)$.
 
\subsubsection{Inner chiral symmetry and time reversal symmetry}

Recall that
insulators which have inner chiral symmetry and time reversal symmetry 
are described by a graded \RA\ $(A,\gamma,\rt)$. If the grading is inner we have two possibilities which are subject to Thm.~\ref{lem-Morita-real}: either the grading operator is time reversal invariant - we said that in this case $\gamma$ is a real inner grading - or it is anti-invariant - the case of an imaginary inner grading. 


In the first case, $\rt(\Gamma) = \Gamma$,  we apply Prop.~\ref{cor-inner-real} to see that
extended topological phases are classified by
$$\DK(A,\gamma,\rt)\cong K_1(A^\rt,\id)\cong KO_1(A^\rt).$$

The same argument as for complex inner chiral symmetry allows us to treat the case in which we have a non-unital algebra with a real inner grading. Indeed, by Thm.~\ref{lem-Morita-real} (3+) we have $(A,\gamma,\rt)\cong (A_{++}\otimes \C l_2,\id\otimes\st,\rt_{++}\otimes\rc_{1,1})$
and so we can identify $h$ with the odd element  $\Psi_e(h)$ of $A_{++}^+\otimes \C l_2$. 
Time reversal invariance now means $\rt_{++,2}(\pi(\Psi_e(h)))  = \pi(\Psi_e(h))$
 and hence $\pi(\Psi_e(\tilde h)) = \begin{pmatrix} 0 & z \\ \bar{z} & 0 \end{pmatrix}$ for 
$z=\pm 1$. We may assume the positive sign, $\pi(\Psi_e(\tilde h)) = \sigma_x$, otherwise changing $h$ to $-h$. In that case $[\tilde h]\in KO_1(A^\rt)$.

  
In the case of inner chiral symmetry with time reversal anti-invariant generator, $\rt(\Gamma) = -\Gamma$, 
Prop.~\ref{cor-inner-real} yields
$$\DK(A,\gamma,\rt)\cong K_{-1}(A^\rt,\id)\cong KO_{-1}(A^\rt).$$

Also in this case 
a non-unital algebra can be handled. The arguments are parallel, except that now 
$(A,\gamma,\rt)\cong (A_{++}\otimes \C l_2,\id\otimes\st,(\Ad_e\circ \rt)_{++}\otimes\rc_{2,0})$
and time reversal invariance means
$(\Ad_{\sigma_x}\circ \rt)_{++,2}(\pi(\Psi_e(h)))  = \pi(\Psi_e(h))$. Contrarily to the real inner graded case   time reversal invariance does not put any further restriction on $\pi(\Psi_e(\tilde h)) = \begin{pmatrix} 0 & z \\ \bar{z} & 0 \end{pmatrix}$ so that we can conclude as in the complex case that $[\tilde h] \in KO_{-1}(A^\rt)$ for non-unital $A$ as well.

We summarize these results in Table~\ref{tab-3}.  
    
\begin{table}[h]
\begin{center}
\caption{Classification for insulators with inner chiral symmetry. Here $\tilde\rt = \Ad_e\circ\rt$.} 
\label{tab-3}
\begin{tabular}{|l|l|l|l|}
\hline
Symmetries & graded algebra & real subalgebra & $K$-group\\
\hline
\hline
only inner chiral & $(A_{++}\!\otimes\! \C l_2,\stot)$ & {\small not applicable} & $KU_1(A)$ \\
\hline
real inner chiral & $(A_{++}\!\otimes\! \C l_2,\stot,\tp\!\otimes\! \rc_{1,1})$ & 
$(A_{++}^\tp\!\otimes\! \C l_{1,1},\stot)$& $KO_1(A^\rt)$\\
imag.\ inner chiral & $(A_{++}\!\otimes\! \C l_2,\stot,\tilde\tp\!\otimes\! \rc_{2,0})$ & 
$(A_{++}^{\tilde \tp}\!\otimes\! \C l_{2,0},\stot)$& $KO_{-1}(A^\rt)$\\
\hline
\end{tabular}
\end{center}
\end{table}
\subsubsection{Classification of inner chiral symmetry w.r.t.\ a reference structure.}
We discuss the cases in which the chiral symmetry $\gamma$ is inner and there is time reversal symmetry which is inner related to a reference structure $\rf$. We assume that there is an 
$\rf$-invariant odd self adjoint unitary $e$ in $\Mm(A)$ 
and furthermore the conditions H1 and H2.\footnote{Note that the existence of $e$ is equivalent to the existence of a $\rt$-invariant self adjoint unitary, as $\Ad_\Theta  (e)$ is $\Ad_\Theta\circ\rf$-invariant if and only if $e$ is $\rf$-invariant.}

The grading is inner, that is, $\gamma = \Ad_\Gamma$, with $\Gamma^2 = 1$. The time reversal symmetry is inner related to $\rf$, i.e.\
$\rt = \Ad_\Theta\circ \rf$, for a homogeneous unitary $\Theta$. Therefore the particle hole symmetry $\rp := \gamma\circ \rt$ can be written 
$$\rp = \Ad_{\Gamma\Theta}\circ \rf$$
and hence we see that its parity relative to $\rf$ is 
$$\eta^{(1)}_{\rp,\rf} = \Gamma\Theta\rf(\Gamma\Theta) = 
\Gamma\Theta \rf(\Gamma) \Theta^{-1}\, \Theta \rf(\Theta) = \Gamma\rt(\Gamma)\eta^{(1)}_{\rt,\rf}.$$
Since $\Gamma\rt(\Gamma)=\pm 1$ depending on whether the grading is real or imaginary inner (for the real structure $\rt$), we see that the parity of PHS equals to the parity of TRS
provided the grading is TRS invariant, and opposite otherwise.
Furthermore
$$ \rt(\Gamma)\Gamma = \rf(\Gamma)\Gamma \eta^{(2)}_{\rt,\rf}.$$
Thus if $\rf$ preserves the grading operator then 
$\eta^{(1)}_{\rt,\rf}\eta^{(2)}_{\rt,\rf}$ is the parity of PHS.
If $\rf$ is defined by (\ref{eq-grf}) then $\rf(\Gamma)=\pm \Gamma$ means that $\pi(\Gamma)$ is a real or an imaginary operator. 
\bigskip

We treat first the case that $\Theta$ is even. We apply
Thm.~\ref{thm-sign-graded} to obtain
\begin{eqnarray*} 
(M_4(A),\gamma_4,\rt_4)&\cong& (M_2(A),\gamma_2,\rf_2)\hat\otimes (M_2(\C),\id,\rs')
\end{eqnarray*}
where $\rs'=\rc_{1,1}$ if $\eta^{(1)}_{\rt,\rf}=+1$ and $\rs'=\rc_{0,2}$ if $\eta^{(1)}_{\rt,\rf}=-1$.
Next we 
apply Thm.~\ref{lem-Morita-real} to obtain
\begin{eqnarray*}
(M_2(A),\gamma_2,\rf_2)\hat\otimes (M_2(\C),\id,\rs')
&\cong &(A, \id,\rf)\hat\otimes 
(\C l_2, \st,\rs'')\hat\otimes (M_2(\C),\id,\rs')\\
&\cong &(A,\id,\rf)\hat\otimes (\C l_4, \st,\rc')
\end{eqnarray*}
where $\rs''=\rc_{1,1}$ if $\rf(\Gamma)\Gamma=1$ 
and $\rs''=\rc_{2,0}$ if $\rf(\Gamma)\Gamma=-1$ and $\rc'$ depends on $\rs',\rs''$ and can be determined by Lemma~\ref{lem-H-times-Cl}.
This yields the following table in which we include the real subalgebra and the $K$-group. Recall that $\Theta$ is even for that table and thus $\rt(\Gamma)\Gamma=\rf(\Gamma)\Gamma$.
\bigskip 

\begin{center}
\begin{tabular}{|c|c|c|c|c|c|c|c|c|c|c|c|c|}
\hline
$\rt(\Theta)\Theta $ &  $\rf(\Gamma)\Gamma$  & $\rs'$ & $\rs''$  & $\rc'$  & real subalg. 
& $K$-group \\
\hline
\hline
 $+1$  & $+1$ & $\rc_{1,1}$ & $\rc_{1,1}$ &
$\rc_{2,2}$ & $A^{\rf}\otimes Cl_{2,2}$ & $KO_1(A^\rf)$ \\
$+1$ & $-1$ &  $\rc_{1,1}$ & $\rc_{2,0}$ &
 $\rc_{3,1}$ & $A^{\rf}\otimes Cl_{3,1}$ & $KO_{-1}(A^\rf)$  \\
$-1$ &  $-1$&  $\rc_{0,2}$ & $\rc_{2,0}$ &
 $\rc_{1,3}$ & $A^{\rf}\otimes Cl_{1,3}$ & $KO_3(A^\rf)$ \\
$-1$ &$+1$ &  $\rc_{0,2}$ & $\rc_{1,1}$ &
 $\rc_{0,4}$  & $A^{\rf}\otimes Cl_{0,4}$ & $KO_5(A^\rf)$ \\
\hline
\end{tabular}
\end{center}

\bigskip

We now consider an odd generator $\Theta$. Thm.~\ref{thm-sign-graded} yields now
\begin{eqnarray*} 
(M_2(A),\gamma_2,\rt_2)&\cong& (A,\gamma,\rf\circ\gamma)\hat\otimes (\C l_2,\st,\rs')
\end{eqnarray*}
where $\rs'=\rc_{2,0}$ if $\eta^{(1)}_{\rt,\rf}=+1$ and $\rs'=\rc_{0,2}$ if $\eta^{(1)}_{\rt,\rf}=-1$.
Note that $\ss:=\rf\circ\gamma = \Ad_\Gamma\circ\rf$ is inner related to $\rf$ with signs
$\eta_{\ss,\rf} = (\Gamma\rf(\Gamma),+1)$. Note that $\tilde e =\begin{pmatrix} 0 & ie \\ -ie & 0 \end{pmatrix}$ is an $\ss$-invariant odd self adjoint unitary in the multiplier algebra of 
$M_2(A)$. We may therefore apply Thm.~\ref{thm-sign-graded} to $(M_2(A),\gamma_2,\ss_2)$ to obtain
$$(M_2(A),\gamma_2,(\rf\circ\gamma)_2) \cong (A,\id,\rf)\hat\otimes (\C l_2,\st,\rs'')$$
with $\rs''=\rc_{1,1}$ if $\Gamma\rf(\Gamma)=+1$ and $\rs''=\rc_{0,2}$ if $\Gamma\rf(\Gamma)=-1$. Hence
\begin{eqnarray*} 
(M_4(A),\gamma_4,\rt_4)&\cong& (A,\gamma,\rf\circ\gamma)\hat\otimes (\C l_2,\st,\rs'')\hat\otimes (\C l_2,\st,\rs')\\
&\cong & (A,\gamma,\rf\circ\gamma)\hat\otimes (\C l_4,\st,\rc')
\end{eqnarray*}
with $\rc'$ depending on $\rs''$ and $\rs'$.
We summarize the results for odd $\Theta$ in the following table.

\bigskip 

\begin{center}
\begin{tabular}{|c|c|c|c|c|c|c|c|c|c|c|c|c|}
\hline
$\rt(\Theta)\Theta $ &  $\rt(\Gamma)\Gamma$ & $\rf(\Gamma)\Gamma$ &  $\rs'$ & $\rs''$ & $\rc'$  & real subalg.& $K$-group  \\
\hline
\hline
 $+1$ &  $+1$ & $-1$  & $\rc_{2,0}$ & 
$\rc_{0,2}$ &$\rc_{2,2}$& $A^{\rf}\otimes Cl_{2,2}$ & $KO_{1}(A^\rf)$\\
$+1$ & $-1$ & $+1$  & $\rc_{2,0}$ & 
$\rc_{1,1}$ & $\rc_{3,1}$ & $A^{\rf}\otimes Cl_{3,1}$& $KO_{-1}(A^\rf)$ \\
$-1$ &  $-1$& $+1$  & $\rc_{0,2}$ & 
$\rc_{1,1}$ & $\rc_{1,3}$& $A^{\rf}\otimes Cl_{1,3}$ & $KO_{3}(A^\rf)$\\
$-1$ & $+1$ & $-1$  & $\rc_{0,2}$ & 
$\rc_{0,2}$ & $\rc_{0,4}$  & $A^{\rf}\otimes Cl_{0,4}$& $KO_{5}(A^\rf)$\\
\hline
\end{tabular}
\end{center}
\medskip

We observe that the
$K$-group of the graded real subalgebra depends only on $\eta^{(1)}_{\rt,\rf}$ and 
$\rt(\Gamma)\Gamma$, or equivalently on the  parities of the real symmetries.
We summarize our results in Table~\ref{tab-4}.  (The grading is called real, if it is TRS invariant.)

\begin{table}[h]
\begin{center}
\caption{Classification for insulators with balanced inner chiral symmetry and real symmetry relative to a reference structure $\rf$.} 
\label{tab-4}
\begin{tabular}{|c|c|c|c|c|c|c|c|c|c|c|c|}
\hline
TRS & PHS & grading & 
$\eta^{(1)}_{\rt,\rf}$ & $\rt(\Gamma)\Gamma$ & 
$K$-group \\
\hline
\hline
even &  even & real & 
$+1$ & $+1$  & 
$KO_1(A^\rf)$\\
even & odd & imag. &
$+1$ &$-1$ & 
$KO_{-1}(A^\rf)$\\
odd & even & imag. &
$-1$& $-1$& 
 $KO_3(A^\rf)$\\
odd & odd & real &
$-1$&$+1$ & 
$KO_5(A^\rf)$\\
\hline
\end{tabular}
\end{center}
\end{table}

\newcommand{\ot}{\otimes}
\subsection{$K$-groups for tight binding models with real symmetry}
As an application we consider the $K$-groups of the observable algebra 
$A =  C(X,M_n(\C))\rtimes_\alpha\Z^d$ and its real subalgebra 
$A^\rf =  C(X,M_n(\R))\rtimes_\alpha\Z^d $
for tight binding models with real symmetry. By stability of the $K$-functor we may assume that 
$n=1$. 
 
The general technique to determine the $K$-theory of a crossed product $B\rtimes_\alpha\Z$
is by means of the Pimsner Voiculescu exact sequence \cite{PV} which has been adapted to the real case in \cite{Schroeder}. This is a  $6$-term or a $24$-term exact sequence in complex or real $K$-theory, respectively, and can be cut into $2$ or $8$ short exact sequences, for each degree $i$ one. 
\begin{equation}\label{eq-PV}
0 \to C_{\alpha} K_i(B) \to K_i(B\rtimes_\alpha\Z) \stackrel{\delta}\to I_{\alpha} K_{i-1}(B) \to 0.
\end{equation}
Here $K_i(B) = K_i(B,\id)$ is either complex or real $K$-theory for the ungraded complex or real \CA\ $B$. Furthermore
$C_{\alpha}$ is the coinvariant functor, that is, for a module $M$ with $\Z$-action $\alpha$,
$C_\alpha M = M /\sim_\alpha$ is the quotient of $M$ by elements of the form $m - \alpha(m)$, 
and $I_\alpha$ the invariant functor, $I_\alpha M  := \{m\in M:\alpha(m)=0\}$.


In our case $A$ can be written as an iterated crossed product: the $\Z^d$-action $\alpha$ is given by $d$ commuting $\Z$-actions $\alpha_1,\cdots,\alpha_d$ and
$A = C(X)\rtimes_{\alpha_d}\Z \cdots \rtimes_{\alpha_1}\Z$. We thus need to iterate the above $d$ times to obtain the $K$-theory of $A$. This becomes quickly complicated, in particular as $I_{\alpha_j} C_{\alpha_k}M$ need not to be the same as $C_{\alpha_k} I_{\alpha_j}M$ and that the exact sequence (\ref{eq-PV}) may not split into a direct sum. However we can make some general remarks.

1) If $X$ is totally disconnected, which is true in our case since we assumed that our configurations have finite local complexity, then $K_i(C(X,\FM))\cong C(X,K_i(\FM))$ 
where $\FM=\C$ or $\FM=\R$. This follows from the fact that compact totally disconnected spaces are inverse limits of finite discrete spaces and the continuity of the $K$ functor.
Under this isomorphism the action of $\Z^d$ on $K_i(C(X,\FM))$ corresponds to its action on $X$. 

2) If $X$ contains a dense orbit then $I_\alpha C(X,K_{i}(\FM)) = K_{i}(\FM)$. 

3) After $d$ fold application of (\ref{eq-PV}) the composition of the quotient maps 
$\delta_j$, $j=1,\cdots,d$, is a surjective map
$$ \tilde \delta = \delta_d\circ\cdots\circ\delta_1: K_i(C(\Omega,\FM)\rtimes_\alpha\Z^d))\to I_\alpha C(X,K_{i-d}(\FM)) .$$ 
An insulator $h$ defines and element $[h]$ of $KU_i(C(\Omega,\C)\rtimes_\alpha\Z^d)$ or
$KU_i(C(\Omega,\R)\rtimes_\alpha\Z^d)$ with degree $i$ depending to its symmetry (see Table~\ref{tab-or},\ref{tab-4}). The strong invariant of $h$ is 
$\tilde \delta[h]\in I_\alpha C(X,K_{i-d}(\FM))$. 
We list below in the case that $X$ has a dense orbit the group 
$I=I_\alpha C(X,K_{i-d}(\FM))$ of strong invariants. This table can be compared with the 
 the famous periodic table established in \cite{Schnyder,Kitaev}.
\begin{table}[h]
\begin{center}
\caption{The group $I$ of strong invariants for $d$-dimensional tight binding models.}
\label{tab-si}
\begin{tabular}{|c|c|c|c|c|c|c|c|c|c|c|c|}
\hline
TRS & PHS & 
$I$ \\
\hline
\hline
even & -& 
$KO_{-d}(\R)$\\
\hline
odd & -& 
$KO_{4-d}(\R)$\\
\hline
-& even &
$KO_{2-d}(\R)$\\
\hline
- & odd &
$KO_{6-d}(\R)$\\
\hline
even &  even & 
$KO_{1-d}(\R)$\\
\hline
even & odd & 
$KO_{7-d}(\R)$\\
\hline
odd & even & 
$KO_{3-d}(\R)$\\
\hline
odd & odd & 
$KO_{5-d}(\R)$\\
\hline
\end{tabular}
\end{center}
\end{table}

4) The quotient map $\delta$ of (\ref{eq-PV}) has a right inverse which is essentially given by the Bott map \cite{PV}. Thus its pre-images may be computed.
In the Kane-Mele model, or any other two dimensional with odd TRS and no PHS the group of strong invariants is $I=KO_2(\R)\cong\Z_2$ (dense orbit assumed). 
Its generator corresponds to the Bott projection over the $2$-sphere, twice the Bott projection being trivial in real the $K$-theory
of the sphere (but not in the complex $K$-theory).

5) For periodic models (band models) $X$ is a single point and the action trivial. The above calculation computes then the complex and real $K_i$-group of the Brillouin zone $\hat{\Z^d}$, with $i$ as in the tables. The genuine Bott map is then a right inverse to $\tilde\delta$ and 
the preimage of $I$ corresponds to the $K_i$-group of the $d$-sphere. This is the invariant Kitaev describes as $\pi_0(R_{i-d})$, the "non weak" invariant \cite{Kitaev}. 
\bigskip


%

6) The periodic table of insulators established in \cite{Schnyder,Kitaev}  can also be obtained in the context of models described by differential operators (continuous models), that is, from the 
algebra $A=M_n(\C)\rtimes_{\id}\R^d$, which arises if $\Omega$ is taken to be a point and there is no magnetic field,  again with 
 reference real structure  $\rf(F)=\overline{F}$ (from Section~\ref{sec-ref}). Indeed, the Fourier transform yields an isomorphism
$(M_n(\C)\rtimes_{\id}\R^d,\rf)\cong (C_0(\R^d,M_n(\C)),\rs)$ where $\rs(\hat F(k)) = \cc(\hat F(-k))$. Furthermore $C_0(\R^d,M_n(\C))^\rs$ is the $n$-fold dual suspension of $M_n(\R)$ 
\cite{Schroeder} and therefore 
$$ KO_i(A^\rf)\cong KO_i(C_0(\R^d,M_n(\C))^\rs) \cong KO_{i-d}(\R).$$



\end{document}